\documentclass[a4paper, 12pt, final]{amsart}
\usepackage{nag}
\usepackage{setspace}
\usepackage{multicol}
\usepackage{multirow}
\usepackage[indentafter]{titlesec}
\usepackage{upgreek}
\usepackage[pdftex]{graphicx}
\usepackage{amsmath, amsfonts, amssymb, amsthm, float, bm}
\usepackage{scalerel}
\usepackage{mathtools}
\usepackage{mathrsfs}
\usepackage{boldline}
\usepackage{enumitem}
\usepackage[toc,page]{appendix}
\usepackage[a4paper, bottom=2cm]{geometry}
\usepackage{array}
\usepackage{xtab}
\usepackage{lmodern}
\usepackage{tcolorbox}
\usepackage{xcolor}
\usepackage{url}
\usepackage{tikz-cd}
\usepackage{bm}
\usepackage{tabularx}
\usepackage{ltablex} 
\usepackage{etoolbox}
\usepackage{tipa}
\usepackage[backend=biber, style=alphabetic]{biblatex}
\usepackage{hyperref}
\hypersetup{colorlinks=true,linkcolor=black,citecolor=blue!80!black}
\usepackage[capitalize]{cleveref}
\renewbibmacro{in:}{}
\DeclareFieldFormat[article]{title}{#1}

\titleformat{name=\section}{}{\thetitle.}{0.8em}{\centering\scshape}
\titleformat{name=\subsection}[runin]{}{\thetitle.}{0.5em}{\bfseries}[.]
\titleformat{name=\subsubsection}[runin]{}{\thetitle.}{0.5em}{\itshape}[.]
\titleformat{name=\paragraph,numberless}[runin]{}{}{0em}{}[.]
\titlespacing{\paragraph}{0em}{0em}{0.5em}
\titleformat{name=\subparagraph,numberless}[runin]{}{}{0em}{}[.]
\titlespacing{\subparagraph}{0em}{0em}{0.5em}

\theoremstyle{plain}
\newtheorem{theorem}{Theorem}[section]
\newtheorem{proposition}[theorem]{Proposition}
\newtheorem{lemma}[theorem]{Lemma}
\newtheorem{corollary}[theorem]{Corollary}

\newtheorem*{theorem*}{Theorem}

\newtheorem*{main}{Main Theorem}

\theoremstyle{definition}
\newtheorem{definition}[theorem]{Definition}

\theoremstyle{remark}
\newtheorem*{remark}{Remark}

\AtBeginEnvironment{theorem}{\setlist[enumerate,1]{label=(\roman*),font=\upshape}}
\AtBeginEnvironment{mtheorem}{\setlist[enumerate,1]{label=(\roman*),font=\upshape}}
\AtBeginEnvironment{theorem*}{\setlist[enumerate,1]{label=(\roman*),font=\upshape}}
\AtBeginEnvironment{ftheorem}{\setlist[enumerate,1]{label=(\roman*),font=\upshape}}
\AtBeginEnvironment{example}{\setlist[enumerate,1]{label=(\roman*),font=\upshape}}
\AtBeginEnvironment{definition}{\setlist[enumerate,1]{label=(\roman*),font=\upshape}}
\AtBeginEnvironment{lemma}{\setlist[enumerate,1]{label=(\roman*),font=\upshape}}
\AtBeginEnvironment{proposition}{\setlist[enumerate,1]{label=(\roman*),font=\upshape}}
\AtBeginEnvironment{sublemma}{\setlist[enumerate,1]{label=(\roman*),font=\upshape}}
\AtBeginEnvironment{corollary}{\setlist[enumerate,1]{label=(\roman*),font=\upshape}}
\AtBeginEnvironment{theorem}{\setlist[enumerate,2]{label=(\alph*),font=\upshape}}
\AtBeginEnvironment{lemma}{\setlist[enumerate,2]{label=(\alph*),font=\upshape}}
\AtBeginEnvironment{proposition}{\setlist[enumerate,2]{label=(\alph*),font=\upshape}}
\AtBeginEnvironment{sublemma}{\setlist[enumerate,2]{label=(\alph*),font=\upshape}}

\renewcommand{\Gamma}{\varGamma}
\renewcommand{\epsilon}{\varepsilon}
\renewcommand{\bar}{\overline}
\renewcommand{\hat}{\widehat}
\renewcommand{\leq}{\leqslant}
\renewcommand{\geq}{\geqslant}

\newcommand{\normaleq}{\trianglelefteq}

\newcommand{\fs}{\mathcal{F}}

\newcommand{\N}{\mathbb{N}}

\newcommand{\SL}{\mathrm{SL}} 
 
\newcommand{\GF}{\mathrm{GF}} 
\newcommand{\syl}{\mathrm{Syl}}
\newcommand{\GL}{\mathrm{GL}}
\newcommand{\Sp}{\mathrm{Sp}}

\newcommand{\PSL}{\mathrm{PSL}}

\newcommand{\PSU}{\mathrm{PSU}}

\newcommand{\Sz}{\mathrm{Sz}}
\newcommand{\Ree}{\mathrm{Ree}}
\newcommand{\Sym}{\mathrm{Sym}}

\newcommand{\Aut}{\mathrm{Aut}}
\newcommand{\Out}{\mathrm{Out}}
\newcommand{\Inn}{\mathrm{Inn}}

\newcommand{\Mor}{\mathrm{Mor}}
\newcommand{\Hom}{\mathrm{Hom}}
\newcommand{\Iso}{\mathrm{Iso}}
\newcommand{\Inj}{\mathrm{Inj}}
\newcommand{\Ob}{\mathrm{Ob}}

\def \wt {\widetilde}

\begin{document}

\title{Fusion Systems on a Sylow $\lowercase{p}$-subgroup of $\mathrm{G}_2(\lowercase{p^n})$ or $\PSU_4(\lowercase{p^n})$}
\author{Martin van Beek }
\thanks{This work formed part of the author's PhD thesis at the University of Birmingham under the supervision of Prof. Chris Parker. The author gratefully acknowledges the financial support received from the EPSRC during this period.}

\begin{abstract}
For any prime $p$ and $S$ a $p$-group isomorphic to a Sylow $p$-subgroup of $\mathrm{G}_2(p^n)$ or $\PSU_4(p^n)$ with $n\in\N$, we determine all saturated fusion systems supported on $S$ up to isomorphism.
\end{abstract}

\maketitle

\section{Introduction}

The purpose of this paper is determine, up to isomorphism, all saturated fusion systems supported on a Sylow $p$-subgroup of $\PSU_4(p^n)$ or $\mathrm{G}_2(p^n)$ for any prime $p$. This work forms part of a program to classify all saturated fusion systems supported on Sylow $p$-subgroups of rank $2$ groups of Lie type, complementing the results in \cite{Clelland} and \cite{Sp4}. Moreover, we generalize results already obtained in \cite{parkersem}, \cite{Baccanelli} and \cite{Raul} where only the case $n=1$ is considered.  Furthermore, we remove some of the other restrictions in those works, where only fusion systems $\fs$ satisfying $O_p(\fs)=\{1\}$ are considered, at little cost to the exposition. The work here draws heavily from results and ideas within those papers and most of the `interesting' examples we uncover occur in this `small' setting.

Additionally, with a small amount of extra effort, for $S$ a Sylow $p$-subgroup of $\PSU_4(p^n)$ or $\mathrm{G}_2(p^n)$, we are able to give a good description of all possible radical, centric subgroups of a fusion system (or group) containing $S$ as a Sylow $p$-subgroup. This has implications outwith the rest of the results in paper. For example, several results concerning weight conjectures for groups and fusion systems rely on detailed information of the radical, centric subgroups of a Sylow $p$-subgroup, see for instance \cite{WeightFusion} and \cite{WeightFusionl}.

Recall that a $\mathcal{K}$-group is a finite group in which all simple sections are known finite simple groups. Although some of the results we apply in this work rely on a $\mathcal{K}$-group hypothesis, within this restricted setting we are almost always able circumvent the need for such a strong assumptions. Where appropriate, we describe the required modifications to make these results independent of a $\mathcal{K}$-group hypothesis. In this way, we are able to almost completely rid ourselves of any reliance on the classification of finite simple groups, and only make use of it to prove the exoticity of some fusion systems supported on a Sylow $7$-subgroup of $\mathrm{G}_2(7)$, a check already completed in \cite{parkersem}, and to recognize $\PSL_2(p^{2n})$ acting on a natural $\Omega_4^-(p^n)$-module to classify fusion system on supported on a Sylow $p$-subgroup of $\PSU_4(p^n)$, where $p$ is odd. We do, however, make use of some of the results listed in \cite{GLS3} concerning known facts about known finite simple groups.

The main theorem is as follows:
\begin{main}\hypertarget{MainTheorem}{}
Suppose that $\fs$ is a saturated fusion system on a $p$-group $S$, where $S$ is isomorphic to a Sylow $p$-subgroup of $\mathrm{G}_2(p^n)$ or $\PSU_4(p^n)$. Then $\fs$ is known. Moreover, if $O_p(\fs)=\{1\}$ then $\fs$ is isomorphic to the $p$-fusion category of an almost simple group; or $p=7$ and $\fs$ is an exotic fusion system on a Sylow $7$-subgroup of $\mathrm{G}_2(7)$. 
\end{main}
In the above classification, more details are given for each prime where they arise in the proofs.

We now describe the strategy to prove the main result of this paper. In Section 2, we set up the requisite group and module theoretic results needed to examine the local actions within a fusion system supported on a $p$-group $S$, where $S$ is isomorphic to a Sylow $p$-subgroup of $\PSU_4(p^n)$ or $\mathrm{G}_2(p^n)$. In Section 3, we provide constructions of Sylow $p$-subgroups of $\mathrm{G}_2(p^n)$ and $\PSU_4(p^n)$ and lay out some important properties of these groups to be utilized in later sections. In Section 4, relevant terminology and results concerning fusion systems is provided, most of which are readily available from standard references. Most importantly, here we describe the necessary tools to describe a complete set of essential subgroups for a saturated fusion system $\fs$ and determine their automizers. The remaining sections deal with the case $\mathrm{G}_2(2^n)$, $\mathrm{G}_2(3^n)$, $\mathrm{G}_2(p^n)$ for $p\geq 5$, and $\PSU_4(p^n)$. For $\mathrm{G}_2(p^n)$, the separation in cases is brought about due to some degeneracies in the Chevalley commutator formulas when $p=2$ or $3$, resulting in some exceptional structural properties. While there are differences when $p=2$ and $p$ is odd for $\PSU_4(p^n)$, the differences are not so drastic to affect our methodology.

In each of the cases, it transpires that, barring some small exceptions, there are only two potential essential subgroups of $\fs$ which coincide with the unipotent radicals of parabolic subgroups in $\mathrm{G}_2(p^n)$ and $\PSU_4(p^n)$. Upon deducing the potential automizers of these subgroups, we then distinguish between the case where there is at most one essential subgroup (where necessarily $O_p(\fs)\ne \{1\}$), and where both subgroups are essential. In this latter case, we apply the main result of \cite{MainThm} which completely determines the fusion system. Indeed, this work as a whole may be viewed as an application of a deeper result recognizing the utility of the amalgam method in fusion systems. The key point is that the uniqueness of the amalgams we extract from the fusion system hypothesis in the main result of \cite{MainThm} implies the uniqueness of the fusion systems. Uniqueness arguments for the local actions in a fusion system are generally the most troublesome checks when classifying certain classes of fusion systems and often the complexity of the arguments tend to scale with the size of the $p$-group which is acted upon.  For the fusion systems classified in this paper, all but a small number arise as the fusion system counterparts to weak BN-pairs of rank $2$, a collection of amalgams whose uniqueness was already verified (at least for the cases relevant to this work) in \cite{Greenbook}. As a result of this observation, we now have a more systematic methodology of treating fusion systems associated to rank $2$ simple groups of Lie type instead of the ad hoc methods used previously.

Importantly within this work, since the only exotic fusion systems we engage with are determined in \cite{parkersem}, we do not need to concern ourselves with checks on saturation and exoticity as in other works. As mentioned throughout, there is some exceptional behaviour for small values of $p$ and $n$ where the fusion systems of some other finite simple groups appear. In these instances, we generally appeal to a package in MAGMA \cite{Comp1} to determine a list of radical, centric subgroups and a list of saturated fusion systems.

Something interesting to note in the above theorem is the small number of exotic fusion systems unearthed. The only exotic fusion systems that arise in the above classification were already identified in \cite{parkersem} and are related to the Monster sporadic simple group. This gives credence to \cite[Conjecture 2]{Comp1} that, aside from a few exceptions in small rank and small prime cases, the structure of a Sylow $p$-subgroup of a group of Lie type in characteristic $p$ is too rigid to support any exotic fusion systems. This is in complete contrast to the cas where the fusion system is supported on a Sylow $p$-subgroup of a group of Lie type in characteristic coprime to $p$, where exotic fusion systems are ubiquitous (see \cite{OliverExotic}).

In terms of progressing towards the goal of determining all fusion systems on Sylow $p$-subgroups of rank $2$ groups of Lie type, this still leaves $\PSU_5(p^n)$, ${}^3\mathrm{D}_4(p^n)$ and ${}^2\mathrm{F}_4(2^n)$, where necessarily $p=2$ in the last case. As in this work, a suitable methodology for classifying fusion systems over the Sylow $p$-subgroups of these groups boils down to determining a complete set of essential subgroups and, after treating small values of $n$ and $p$ separately, applying the main theorem of \cite{MainThm}. 

Our notation and terminology for groups is reasonably standard and generally follows that used in \cite{asch2}, \cite{gor}, \cite{kurz}, \cite{Huppert} and in the ATLAS \cite{atlas}. For fusion systems, we follow the notation used in \cite{ako}. Some clarification is probably needed on the notation we use for group extensions. We use $A:B$ to denote the semidirect product of $A$ and $B$, where $A$ is normalized by $B$. We use the notation $A.B$ to denote an arbitrary extension of $B$ by $A$. That is, $A$ is a normal subgroup of $A.B$ such that the quotient of $A.B$ by $A$ is isomorphic to $B$. We use the notation $A\cdot B$ to denote a central extension of $B$ by $A$ and the notation $A*B$ to denote a central product of $A$ and $B$, where the intersection of $A$ and $B$ will be clear whenever this arises.

\section{Preliminaries: Group and Module Theory}\label{GrpSec}

We recall some useful facts concerning groups and modules which we employ later in the paper.

\begin{definition}
A finite group $G$ is a \emph{$\mathcal{K}$-group} if every simple section of $G$ is a known finite simple group.
\end{definition}

\begin{definition}
Let $G$ be a finite group and $p$ a prime dividing $|G|$. Then $G$ is of \emph{characteristic $p$} if $C_G(O_p(G))\le O_p(G)$. Equivalently, $F^*(G)=O_p(G)$.
\end{definition}

We will often use the results described below without explicit reference, and where we do reference, we will refer to the totality of the techniques as ``coprime action.''

\begin{lemma}[Coprime Action]
Suppose that $G$ acts on $A$ coprimely. That is, $|G|$ is coprime to $|A|$. Let $B$ be a $G$-invariant normal subgroup of $A$. Then the following holds:
\begin{enumerate}
\item $C_{A/B}(G)=C_A(G)B/B$;
\item if $G$ acts trivially on $A/B$ and $B$, then $G$ acts trivially on $A$;
\item $[A, G]=[A,G,G]$;
\item $A=[A,G]C_A(G)$ and if $A$ is abelian $A=[A,G]\times C_A(G)$;
\item if $G$ acts trivially on $A/\Phi(A)$, then $G$ acts trivially on $A$;
\item if $A$ is a $p$-group and $G$ acts trivially on $\Omega(A)$, then $G$ acts trivially on $A$; and
\item for $S\in\syl_p(G)$, if $m_p(S)\geq 2$ then $A=\langle C_A(s) \mid s\in S\setminus\{1\}\rangle$.
\end{enumerate}
\end{lemma}
\begin{proof}
See \cite[Chapter 8]{kurz}.
\end{proof}

We present two further lemmas which loosely fall under the umbrella of coprime action.

\begin{lemma}\label{burnside}
Let $S$ be a finite $p$-group. Then $C_{\Aut(S)}(S/\Phi(S))$ is a normal $p$-subgroup of $\Aut(S)$.
\end{lemma}
\begin{proof}
This is due to Burnside, see \cite[Theorem 5.1.4]{gor}.
\end{proof}

\begin{lemma}\label{GrpChain}
Let $E$ be a finite $p$-group and $Q\le A$ where $A\le \Aut(E)$ and $Q$ is a $p$-group. Suppose there exists a normal chain $\{1\} =E_0 \normaleq E_1  \normaleq E_2 \normaleq \dots \normaleq E_m = E$ of subgroups such that for each $\alpha \in A$, $E_i\alpha = E_i$ for all $0 \le i \le m$. If for all $1\le i\le m$, $Q$ centralizes $E_i/E_{i-1}$, then $Q\le O_p(A)$.
\end{lemma}
\begin{proof}
See \cite[{(I.5.3.2)}]{gor}.
\end{proof}

Pivotal to the analysis of local actions within a fusion system is recognizing $\SL_2(p^n)$ acting on its modules in characteristic $p$. In studying Sylow $p$-subgroups of $\mathrm{G}_2(p^n)$ and $\PSU_4(p^n)$, the most important characterization will be $\SL_2(p^n)$ acting on its \emph{natural module}.

\begin{definition}
A \emph{natural $\SL_2(p^n)$-module} is any irreducible $2$-dimensional \linebreak $\GF(p^n)\SL_2(p^n)$-module regarded as a $2n$-dimension module for $\GF(p)\SL_2(p^n)$ by restriction.
\end{definition}

\begin{lemma}\label{sl2p-mod}
Suppose $G\cong\SL_2(p^n)$, $S\in\syl_p(G)$ and $V$ is natural $\SL_2(p^n)$-module. Then the following holds:
\begin{enumerate}
\item $[V, S,S]=\{1\}$;
\item $|V|=p^{2n}$ and $|C_V(S)|=p^n$;
\item $C_V(s)=C_V(S)=[V,S]=[V,s]=[v, S]$ for all $v\in V\setminus C_V(S)$ and $1\ne s\in S$;
\item $V=C_V(S)\times C_V(S^g)$ for $g\in G\setminus N_G(S)$;
\item every $p'$-element of $G$ acts fixed point freely on $V$; and
\item $V/C_V(S)$ and $C_V(S)$ are irreducible $\mathrm{GF}(p)N_G(S)$-modules upon restriction.
\end{enumerate}
\end{lemma}
\begin{proof}
See \cite[Lemma 4.6]{parkerBN}
\end{proof}

More generally, a natural $\SL_2(p^n)$-module provides the minimal example of a non-trivial failure to factorize module. 

\begin{definition}
Let $G$ be a finite group, $V$ a $\GF(p)G$-module and $A\le G$. If 
\begin{enumerate}
\item $A/C_A(V)$ is an elementary abelian $p$-group;
\item $[V,A]\ne\{1\}$; and 
\item $|V/C_V(A)|\leq |A/C_A(V)|$
\end{enumerate}
then $V$ is a \emph{failure to factorize module} (abbrev. FF-module) for $G$ and $A$ is an \emph{offender} on $V$. 
\end{definition}

FF-modules are named due to how they arise as counterexamples to \emph{Thompson factorization} (see \cite[{{32.11}}]{asch2}), which aims to factorize a group into two $p$-local subgroups. One of these $p$-local subgroups is the normalizer of the Thompson subgroup of a fixed Sylow $p$-subgroup. Independent of FF-modules, the Thompson subgroup is incredibly useful in studying the structure of a $p$-group and will play an important role in the analysis of subgroups of Sylow $p$-subgroups of $\mathrm{G}_2(p^n)$ and $\PSU_4(p^n)$ later.

\begin{definition}
Let $S$ be a finite $p$-group. Set $\mathcal{A}(S)$ to be the set of all elementary abelian subgroup of $S$ of maximal rank. Then the \emph{Thompson subgroup} of $S$ is defined as $J(S):=\langle A \mid A\in\mathcal{A}(S)\rangle$.
\end{definition}

\begin{proposition}
Let $S$ be a finite $p$-group. Then the following holds:
\begin{enumerate}
\item $J(S)$ is a characteristic subgroup of $S$;
\item $C_S(J(S))\le J(S)$; and
\item if $J(S)\le T\le S$, then $J(S)=J(T)$.
\end{enumerate}
\end{proposition}
\begin{proof}
See \cite[{{9.2.8}}]{kurz}.
\end{proof}

The following proposition describes a fairly natural situation in which one can identify an FF-module from a group failing to satisfy Thompson factorization. This result is well known and the proof is standard (see \cite[{{9.2}}]{kurz}).

\begin{proposition}\label{BasicFF}
Let $G=O^{p'}(G)$ be a finite group with $S\in\syl_p(G)$ and $F^*(G)=O_p(G)$. Set $V:=\langle \Omega(Z(S))^G\rangle$. If $\Omega(Z(S))<V$ then $O_p(G)=O_p(C_G(V))$ and $O_p(G/C_G(V))=\{1\}$. Furthermore, if $J(S)\not\le C_S(V)$ then $V$ is an FF-module for $G/C_G(V)$.
\end{proposition}

We will also be interested in the $\Omega_4^-(p^n)$-module associated to $\SL_2(p^{2n})$. This arises from the isomorphism $\PSL_2(p^{2n})\cong\Omega_4^-(p^n)$ and will feature predominantly in the analysis of Sylow $p$-subgroups of $\PSU_4(p^n)$. When $p=2$, since $\SL_2(p^n)$ now has more than one ``natural module'', we will need to explicitly distinguish between these two modules.

\begin{definition}
Let $V$ be a natural $\SL_2(p^{2n})$-module for $G\cong \SL_2(p^{2n})$. A \emph{natural $\Omega_4^-(p^n)$-module} for $G$ is any non-trivial irreducible submodule of $(V\otimes_k V^\tau)_{\GF(p^n)G}$ regarded as a $\GF(p)G$-module by restriction, where $\tau$ is an involutary automorphism of $\GF(p^{2n})$.
\end{definition}

\begin{lemma}\label{Omega4}
Let $G\cong \mathrm{(P)SL}_2(p^n)$, $S\in\syl_p(G)$ and $S\in\syl_p(G)$ and $V$ a natural $\Omega_4^-(p^n)$-module for $G$. Then the following holds:
\begin{enumerate}
\item $C_G(V)=Z(G)$;
\item $[V, S,S, S]=\{1\}$;
\item $|V|=p^{4n}$ and $|V/[V,S]|=|C_V(S)|=p^n$;
\item $|C_V(s)|=|[V,s]|=p^{2n}$ and $[V,S]=C_V(s)\times [V,s]$ for all $1\ne s\in S$; and
\item $V/[V,S]$ and $C_V(S)$ are irreducible $\mathrm{GF}(p)N_G(S)$-modules upon restriction.
\end{enumerate}
Moreover, for $\{1\}\ne F\le S$, one of the following occurs:
\begin{enumerate}[label=(\alph*)]
\item $[V, F]=[V, S]$ and $C_{V}(F)=C_{V}(S)$;
\item $p=2$, $[V, F]=C_{V}(F)$ has order $p^{2n}$, $F$ is quadratic on $V$ and $|F|\leq p^n$; or
\item $p$ is odd, $|[V, F]|=|C_{V}(F)|=p^{2n}$, $[V, S]=[V, F]C_{V}(F)$, $C_V(S)=C_{[V, F]}(F)$ and $|F|\leq p^n$.
\end{enumerate}
\end{lemma}
\begin{proof}
See \cite[Lemma 4.8]{parkerBN} and \cite[Lemma 3.15]{parkerSymp}.
\end{proof}

In addition to the structural information regarding natural $\SL_2(p^n)$-modules and natural $\Omega_4^-(p^n)$-modules, we also require methods to recognize these actions in local settings.

\begin{lemma}\label{DirectSum}
Let $G\cong\SL_2(p^n)$ and $S\in\syl_p(G)$. Suppose that $V$ is a module for $G$ over $\GF(p)$ such that $[V,S,S]=\{1\}$ and $[V, O^p(G)]\ne\{1\}$. Then $[V/C_V(O^p(G)), O^p(G)]$ is a direct sum of natural modules for $G$.
\end{lemma}
\begin{proof}
See \cite[Lemma 2.2]{ChermakQuad}.
\end{proof}

\begin{lemma}\label{SL2ModRecog}
Let $G$ be a $p'$-central extension of $\PSL_2(p^n)$, $S\in\syl_p(G)$ and $V$ a faithful irreducible $\GF(p)$-module. If $|V|=p^{2n}$ then either
\begin{enumerate}
\item $V$ is a natural $\SL_2(p^n)$-module for $G\cong \SL_2(p^n)$; or
\item $V$ is a natural $\Omega_4^-(p^{n/2})$-module, $n$ is even, $S$ does not act quadratically on $V$ and $Z(G)$ acts trivially on $V$.
\end{enumerate}
\end{lemma}
\begin{proof}
See \cite[Lemma 2.6]{ChermakJ}.
\end{proof}
 
\section[Sylow \texorpdfstring{$p$}{p}-subgroups of \texorpdfstring{$\mathrm{G}_2(p^n)$}{G2(pn)} and \texorpdfstring{$\PSU_4(p^n)$}{PSU4(pn)}]{Sylow $p$-subgroups of $\mathrm{G}_2(p^n)$ and $\PSU_4(p^n)$}\label{G2Sylow}

In this section, we construct Sylow $p$-subgroups of $\mathrm{G}_2(p^n)$ and $\PSU_4(p^n)$ and describe some of their basic properties. We refer to \cite{Carter} for constructions and properties of $\mathrm{G}_2(p^n)$ and $\PSU_4(p^n)$, as well as generic properties and terminology regarding the simple groups of Lie type. Throughout, we set $q=p^n$.

We present the root system of type $\mathrm{G}_2$ below. We follow the choices of roots as in \cite[443]{ree} and depict a slightly altered root system than what is given in that paper \cite[Figure 1]{ree}.

\begin{center}
\begin{tikzpicture}
    \foreach\ang in {60,120,...,360}{
     \draw[->,thick] (0,0) -- (\ang:2cm);
    }
    \foreach\ang in {30,90,...,330}{
     \draw[->,thick] (0,0) -- (\ang:3cm);
    }
    
    \node[anchor = center,scale=0.8] at (0,3.3) {$\beta$};
    \node[anchor = center,scale=0.8] at (1.2,2.2) {$\alpha+\beta$};
    \node[anchor = center,scale=0.8] at (3.3,1.5) {$3\alpha+2\beta$};
    \node[anchor = center,scale=0.8] at (2.8,0) {$2\alpha+\beta$};
    \node[anchor = center,scale=0.8] at (3.3,-1.5) {$3\alpha+\beta$};
    \node[anchor = center,scale=0.8] at (1.2,-2.2) {$\alpha$};

    \node[anchor = center,scale=0.8] at (0,-3.3) {-$\beta$};
    \node[anchor = center,scale=0.8] at (-1.2,-2.2) {$-(\alpha+\beta)$};
    \node[anchor = center,scale=0.8] at (-3.5,-1.5) {$-(3\alpha+2\beta)$};
    \node[anchor = center,scale=0.8] at (-3.0,0) {$-(2\alpha+\beta)$};
    \node[anchor = center,scale=0.8] at (-3.5,1.5) {$-(3\alpha+\beta)$};
    \node[anchor = center,scale=0.8] at (-1.2, 2.2) {$-\alpha$};
\end{tikzpicture}
\end{center}

In this way, we can arrange that our six positive roots are \[\Phi^+=\{\alpha, \beta, \alpha+\beta, 2\alpha+\beta, 3\alpha+\beta, 3\alpha+2\beta\}.\] For $\epsilon\in\Phi^+$ we set $X_\epsilon:=\langle x_\epsilon(t) \mid t\in \mathbb{K}\rangle$, where $\mathbb{K}$ is a field of order $q$. Thus, we have that \[S=\langle X_{\alpha},  X_\beta,  X_{3\alpha+\beta}, X_{\alpha+\beta}, X_{2\alpha+\beta}, X_{3\alpha+2\beta}\rangle\in\syl_p(\mathrm{G}_2(q))\] is of order $q^6$.

Using results from \cite[(3.10)]{ree}, we have the following Chevalley commutator formulas for the root subgroups:
\begin{align*}
[x_\alpha(t), x_\beta(u)]&=x_{\alpha+\beta}(-tu)x_{2\alpha+\beta}(-t^2u)x_{3\alpha+\beta}(t^3u)x_{3\alpha+2\beta}(-2t^3u^2)\\
[x_\alpha(t), x_{\alpha+\beta}(u)]&=x_{2\alpha+\beta}(-2tu)x_{3\alpha+\beta}(3t^2u)x_{3\alpha+2\beta}(3tu^2)\\
[x_\alpha(t), x_{2\alpha+\beta}(u)]&=x_{3\alpha+\beta}(3tu)\\
[x_\beta(t), x_{3\alpha+\beta}(u)]&=x_{3\alpha+2\beta}(tu)\\
[x_{\alpha+\beta}(t), x_{2\alpha+\beta}(u)]&=x_{3\alpha+2\beta}(3tu).
\end{align*}

We remark that the coefficients in the commutator formulas showcase obvious degeneracies when $p\in\{2,3\}$. This is one of the reasons we treat these cases separately.

\begin{lemma}\label{G2Exponent}
Suppose that $S$ is isomorphic to a Sylow $p$-subgroup of $\mathrm{G}_2(q)$. Then the following holds:
\begin{enumerate}
\item if $p=2$, then $S$ has exponent $8$;
\item if $p\in\{3,5\}$, then $S$ has exponent $p^2$; and
\item if $p\geq 7$, then $S$ has exponent $p$. 
\end{enumerate}
\end{lemma}
\begin{proof}
Since $\mathrm{G}_2(q)$ has a $7$ dimensional representation over $\GF(q)$ when $p$ is odd, and $\mathrm{G}_2(q)$ has a $6$ dimensional representation over $\GF(q)$ when $p=2$, we can find an upper bound for the exponent of $S$ by calculating the exponent of a Sylow $p$-subgroup of $\GL_r(q)$, where $r=7$ when $p$ is odd and $r=6$ if $p=2$. But a Sylow $p$-subgroup of $\GL_r(q)$ has exponent $p^a$ with $a$ minimal such that $p^a>r-1$. Thus, $S$ has exponent $p$ when $p\geq 7$ and the exponent of $S$ is bounded above by $p^2$ or $8$ when $p\in\{3,5\}$ or $p=2$ respectively. One can compute directly that a Sylow $p$-subgroup of $\mathrm{G}_2(p)$ has exponent $8$, $9$ or $25$ when $p=2$,$3$ or $5$ respectively, and so the result follows.
\end{proof}

We now proceed with the construction of a Sylow $p$-subgroup $S$ of $\PSU_4(q)\cong {}^2\mathrm{A}_3(q^2)$. Let $\Phi^+=\{a, b, c, a+b, a+c, b+c, a+b+c\}$ be a choice of positive roots for the root system $\mathrm{A}_3$. In particular, under the symmetry of $\mathrm{A}_3$, we may partition the positive roots into equivalence classes $\{a, c\}$, $\{b\}$, $\{a+b, b+c\}$ and $\{a+b+c\}$. Following \cite[Theorem 2.4.1]{GLS3} and setting $\hat{\mathbb{K}}$ to be a finite field of order $q^2$, and $\mathbb{K}$ the subfield of order $q$, we may choose a set of fundamental roots $\{\alpha, \beta\}$ for ${}^2\mathrm{A}_3(q^2)$ as
\begin{align*}
x_{\alpha}(t)&=x_a(t)x_c(t^q), \\
x_{\beta}(u)&=x_b(u),
\end{align*}
where $t,u\in\hat{\mathbb{K}}$ and $u=u^q\in\mathbb{K}$. We then retrieve a full set of positive roots and root subgroups for ${}^2\mathrm{A}_3(q^2)$:
\begingroup
\allowdisplaybreaks
\begin{align*}
x_{\alpha}(t)&=x_a(t)x_c(t^q), \\
x_{\beta}(u)&=x_b(u),\\
x_{\alpha+\beta}(t)&=x_{a+b}(t)x_{b+c}(t^q), \\
x_{2\alpha+\beta}(u)&=x_{a+b+c}(u)
\end{align*}
\endgroup
where $t,u\in\hat{\mathbb{K}}$ and $u=u^q\in\mathbb{K}$. Hence, we infer that 
\[|X_{\alpha}|=q^2,\,\,
|X_{\beta}|=q,\,\,
|X_{\alpha+\beta}|=q^2,\,\,
|X_{2\alpha+\beta}|=q\] and $S=\langle X_\alpha, X_\beta, X_{\alpha+\beta}, X_{2\alpha+\beta}\rangle$ is of order $q^6$. 

We reproduce the Chevalley commutator formulas for ${}^2\mathrm{A}_3(q^2)$ and as, before, set $\mathbb{K}$ to be a field of order $q$. For more details, see \cite[Theorem 2.4.5]{GLS3}.
\begin{align*}
[x_\alpha(t), x_\beta(u)]&=x_{\alpha+\beta}(\epsilon tu)x_{2\alpha+\beta}(\epsilon' N(t)u)\\
[x_\alpha(t), x_{\alpha+\beta}(u)]&=x_{2\alpha+\beta}(\epsilon'' Tr(tu))
\end{align*}
where $t,u\in\hat{\mathbb{K}}$ and $u=u^q$, and $Tr$ and $N$ denote the field trace and norm from $\hat{\mathbb{K}}$ down to $\mathbb{K}$. Moreover, $\epsilon, \epsilon', \epsilon''\in\{1,-1\}$ depend only on the roots in the commutators they are involved in. It then follows that
\begin{align*}
J(S)&=X_{\beta}X_{\alpha+\beta}X_{2\alpha+\beta},\\
S'&=X_{\alpha+\beta}X_{2\alpha+\beta},\\
Z(S)&=X_{2\alpha+\beta}.
\end{align*}
For the purposes here, the exact values of $\epsilon, \epsilon'$ and $\epsilon''$ are not important and all we require is that commutators with single elements generate entire $\mathrm{GF}(q)$ spaces of root subgroups e.g. $[x_{\alpha}(t), S']=Z(S)$ and $|[x_{\alpha}(t), J(S)]|=q^2$ for all $t\ne 0$.

In the analysis of $S\in\syl_p(\PSU_4(q))$, it will often be more useful to work with local subgroups of $\PSU_4(q)$, recognizing the internal modules within these local subgroups and obtaining information about $S$ from its embedding in these groups. In this way, we work with the elements as matrices explicitly, recognizing the isomorphism ${}^2\mathrm{A}_3(q^2)\cong\PSU_4(q)\le \PSL_4(q^2)$. However, for some arguments, we still reference the commutator formulas.

\begin{lemma}\label{PSUExponent}
Suppose that $S$ is isomorphic to a Sylow $p$-subgroup of $\PSU_4(q)$. Then the following holds:
\begin{enumerate}
\item if $p=2$, then $S$ has exponent $4$;
\item if $p=3$, then $S$ has exponent $9$; and
\item if $p\geq 5$, then $S$ has exponent $p$. 
\end{enumerate}
\end{lemma}
\begin{proof}
This proof is much the same as \cref{G2Exponent}. Since $\PSU_4(q)$ is a subgroup of $\PSL_4(q^2)$, we can find an upper bound for the exponent of $S$ by calculating the exponent of a Sylow $p$-subgroup of $\GL_4(q^2)$, which is  $p^a$ with $a$ minimal such that $p^a>3$. Thus, $S$ has exponent $p$ when $p\geq 5$ and the exponent of $S$ is bounded above by $4$ or $9$ when $p=2$ or $p=3$ respectively. One can compute directly that a Sylow $p$-subgroup of $\PSU_4(p)$ has exponent $p^2$ when $p\in\{2,3\}$, and so the result holds.
\end{proof}

\section{Fusion Systems}

In the section, we set various notations and provide several results concerning fusion systems. As remarked earlier, most of what is written is here is fairly standard and may be extracted from \cite{craven} and \cite{ako}. 

\begin{definition}
Let $G$ be a finite group with $S\in\syl_p(G)$. The \emph{fusion category} of $G$ over $S$, written $\fs_S(G)$, is the category with object set $\Ob(\fs_S(G)):= \{Q: Q\le S\}$ and for $P,Q\le S$, $\Mor_{\fs_S(G)}(P,Q):=\Hom_G(P,Q)$, where $\Hom_G(P,Q)$ denotes maps induced by conjugation by elements of $G$. That is, all morphisms in the category are induced by conjugation by elements of $G$.
\end{definition}

\begin{definition}
Let $S$ be a $p$-group. A fusion system $\fs$ over $S$ is a category with object set $\Ob(\fs):=\{Q: Q\le S\}$ and whose morphism set satisfies the following properties for $P, Q\le S$:
\begin{itemize}
\item $\Hom_S(P, Q)\subseteq \Mor_{\fs}(P,Q)\subseteq \Inj(P,Q)$; and
\item each $\phi\in\Mor_{\fs}(P,Q)$ is the composite of an $\fs$-isomorphism followed by an inclusion,
\end{itemize}
where $\Inj(P,Q)$ denotes injective homomorphisms between $P$ and $Q$. To motivate the group analogy, we write $\Hom_{\fs}(P,Q):=\Mor_{\fs}(P,Q)$ and $\Aut_{\fs}(P):=\Hom_{\fs}(P,P)$.

Two subgroups of $S$ are said to be \emph{$\fs$-conjugate} if they are isomorphic as objects in $\fs$. We write $Q^{\fs}$ for the set of all $\fs$-conjugates of $Q$. We say a fusion system is \emph{realizable} if there exists a finite group $G$ with $S\in\syl_p(G)$ and $\fs=\fs_S(G)$. Otherwise, the fusion system is said to be \emph{exotic}.
\end{definition}

\begin{definition}
Let $\fs$ be a fusion system on a $p$-group $S$. Then $\mathcal{H}$ is a \emph{subsystem} of $\fs$, written $\mathcal{H}\le \fs$, on a $p$-group $T$ if $T\le S$, $\mathcal{H}\subseteq \fs$ as sets and $\mathcal{H}$ is itself a fusion system. Then, for $\fs_1, \fs_2$ subsystems of $\fs$, write $\langle \fs_1, \fs_2\rangle$ for the smallest subsystem of $\fs$ containing $\fs_1$ and $\fs_2$.
\end{definition}

We now give a brief overview of some key concept and definitions within the study of fusion systems. 

\begin{definition}
Let $\fs$ be a fusion system over a $p$-group $S$ and let $Q\le S$. Say that $Q$ is
\begin{itemize}
\item \emph{fully $\fs$-normalized} if $|N_S(Q)|\ge |N_S(P)|$ for all $P\in Q^{\fs}$;
\item \emph{fully $\fs$-centralized} if $|C_S(Q)|\ge |C_S(P)|$ for all $P\in Q^{\fs}$;
\item \emph{fully $\fs$-automized} if $\Aut_S(Q)\in\syl_p(\Aut_{\fs}(Q))$;
\item \emph{receptive} in $\fs$ if for each $P\le S$ and each $\phi\in\Iso_{\fs}(P,Q)$, setting \[N_{\phi}=\{g\in N_S(P) : {}^{\phi}c_g\in\Aut_S(Q)\},\] there is $\bar{\phi}\in\Hom_{\fs}(N_{\phi}, S)$ such that $\bar{\phi}|_P = \phi$;
\item \emph{$S$-centric} if $C_S(Q)=Z(Q)$ and \emph{$\fs$-centric} if $P$ is $S$-centric for all $P\in Q^{\fs}$;
\item \emph{$S$-radical} if $O_p(\Out(Q))\cap \Out_S(Q)=\{1\}$;
\item \emph{$\fs$-radical} if $O_p(\Out_{\fs}(Q))=\{1\}$; or
\item \emph{$\fs$-essential} if $Q$ is $\fs$-centric, fully $\fs$-normalized and $\Out_{\fs}(Q)$ contains a strongly $p$-embedded subgroup.
\end{itemize}
\end{definition}

If it is clear which fusion system we are working in, we will refer to subgroups as being fully normalized (centralized, centric etc.) without the $\fs$ prefix.

For a fusion system $\fs$, we set $\mathcal{E}(\fs)$ to be the set of essential subgroups of $\fs$ and note that essential subgroups of $S$ are fully $\fs$-normalized, $\fs$-centric, $\fs$-radical subgroups by definition. We also remark that any fully $\fs$-normalized, $\fs$-radical subgroup is also $S$-radical.

We mostly care about \emph{saturated} fusion systems as they most closely parallel groups and have the most interesting applications.

\begin{definition}
Let $\fs$ be a fusion system over a $p$-group $S$. Then $\fs$ is \emph{saturated} if the following conditions hold:
\begin{enumerate}
\item Every fully $\fs$-normalized subgroup is also fully $\fs$-centralized and fully $\fs$-automized.
\item Every fully $\fs$-centralized subgroup is receptive in $\fs$.
\end{enumerate}
By a theorem of Puig \cite{Puig1}, the fusion category of a finite group $\fs_S(G)$ is a saturated fusion system.
\end{definition}

From this point on, we implicitly assume that the fusion systems we study are \emph{saturated}, although some of the results we describe apply in wider contexts and can even be used to determine whether or not a fusion system is saturated.

\begin{definition}
A \emph{local $\mathcal{CK}$-system} is a saturated fusion system $\fs$ on a $p$-group $S$ such that $\Aut_{\fs}(P)$ is a $\mathcal{K}$-group for all $P\le S$.
\end{definition}

Local $\mathcal{CK}$-systems provide a means to apply group theoretic results which rely on a $\mathcal{K}$-group hypothesis. This allows for minimal counterexample arguments in fusion systems and provides a link between fusion systems and the classification of finite simple groups. That is, if $G$ is a finite group which is a counterexample to the classification with $|G|$ minimal subject to these constraints, then $\fs_S(G)$ is a local $\mathcal{CK}$-system for $S\in\syl_p(G)$.

We now present arguably the most important tool in classifying saturated fusion systems. Because of this, we need only investigate the local action on a relatively small number of $p$-subgroups to obtain a global characterization of a saturated fusion system.

\begin{theorem}[Alperin -- Goldschmidt Fusion Theorem]
Let $\fs$ be a saturated fusion system over a $p$-group $S$. Then \[\fs=\langle \Aut_{\fs}(Q) \mid Q\,\, \text{is essential or}\,\, Q=S \rangle.\]
\end{theorem}
\begin{proof}
See \cite[Theorem I.3.5]{ako}.
\end{proof}

Along these lines, another important notion is for a $p$-subgroup to be \emph{normal} in a saturated fusion system.

\begin{definition}
Let $\fs$ be a fusion systems over a $p$-group $S$ and $Q\le S$. Say that $Q$ is \emph{normal} in $\fs$ if $Q\normaleq S$ and for all $P,R\le S$ and $\phi\in\Hom_{\fs}(P,R)$, $\phi$ extends to a morphism $\bar{\phi}\in\Hom_{\fs}(PQ,RQ)$ such that $\bar{\phi}(Q)=Q$.

It may be checked that the product of normal subgroups is itself normal. Thus, we may talk about the largest normal subgroup of $\fs$ which we denote $O_p(\fs)$ (and occasionally refer to as the $p$-core of $\fs$).  Further, it follows immediately from the saturation axioms that any subgroup normal in $S$ is fully normalized and fully centralized.
\end{definition}

\begin{definition}
Let $\fs$ be a fusion system over a $p$-group $S$ and let $Q$ be a subgroup. The \emph{normalizer fusion subsystem} of $Q$, denoted $N_{\fs}(Q)$, is the largest subsystem of $\fs$, supported over $N_S(Q)$, in which $Q$ is normal. 
\end{definition}

It is clear from the definition that if $\fs$ is the fusion category of a group $G$ i.e. $\fs=\fs_S(G)$, then $N_{\fs}(Q)=\fs_{N_S(Q)}(N_G(Q))$. The following result is originally attributed to Puig \cite{Puig2}.

\begin{theorem}
Let $\fs$ be a saturated fusion system over a $p$-group $S$. If $Q\le S$ is fully $\fs$-normalized then $N_{\fs}(Q)$ is saturated.
\end{theorem}
\begin{proof}
See \cite[Theorem I.5.5]{ako}.
\end{proof}

\begin{definition}
Let $\fs$ be a fusion systems over a $p$-group $S$ and $P\le Q\le S$. Say that $P$ is \emph{$\fs$-characteristic} in $Q$ if $\Aut_{\fs}(Q)\le N_{\Aut(Q)}(P)$.
\end{definition}

Plainly, if $Q\normaleq\fs$ and $P$ is $\fs$-characteristic in $Q$, then $P\normaleq\fs$.

We now present a link between normal subgroups of a saturated fusion system $\fs$ and its essential subgroups.

\begin{proposition}\label{normalinF}
Let $\fs$ be a saturated fusion system over a $p$-group $S$. Then $Q$ is normal in $\fs$ if and only if $Q$ is contained in each essential subgroup, $Q$ is $\Aut_{\fs}(E)$-invariant for any essential subgroup $E$ of $\fs$ and $Q$ is $\Aut_{\fs}(S)$-invariant.
\end{proposition}
\begin{proof}
See \cite[Proposition I.4.5]{ako}.
\end{proof}

In order to investigate the local actions in a saturated fusions, and in particular in its normalizer subsystems, it will often be convenient to work in a purely group theoretic context. The \emph{model theorem} guarantees that we may do this for a certain class of $p$-subgroups of a saturated fusion system $\fs$.

\begin{theorem}[Model Theorem]\label{model}
Let $\fs$ be a saturated fusion system over a $p$-group $S$. Fix $Q\le S$ which is $\fs$-centric and normal in $\fs$. Then the following holds:
\begin{enumerate} 
\item There are models for $\fs$.
\item If $G_1$ and $G_2$ are two models for $\fs$, then there is an isomorphism $\phi: G_1\to G_2$ such that $\phi|_S = \mathrm{Id}_S$.
\item For any finite group $G$ containing $S$ as a Sylow $p$-subgroup such that $Q\le G$, $C_G(Q) \le Q$ and $Aut_G(Q) = Aut_{\fs}(Q)$, there is $\beta\in\Aut(S)$ such that $\beta|_Q = \mathrm{Id}_Q$ and $\fs_S(G) = {}^{\beta}\fs$. Thus, there is a model for $\fs$ which is isomorphic to $G$.
\end{enumerate}
\end{theorem}
\begin{proof}
See \cite[Theorem I.4.9]{ako}.
\end{proof}

Fusion systems satisfying the hypothesis of the above theorem are referred to as \emph{constrained} fusion systems. It is clear that if $E$ is an essential subgroup of $\fs$,$E$ is a centric normal subgroup of $N_{\fs}(E)$, $N_{\fs}(E)$ is constrained and there is a model $G$ for $N_{\fs}(E)$ with $O_p(G)=E$.

With the aim of applying the Alperin--Goldschmidt fusion theorem, we present the following lemmas which provide the main tools for determining whether a $p$-group is an essential subgroup of saturated fusion system $\fs$.

\begin{lemma}\label{Chain}
Let $E$ be a finite $p$-group and $Q\le A$ where $A\le \Aut(E)$ and $Q$ is a $p$-group. Suppose there exists a normal chain $\{1\} =E_0 \normaleq E_1  \normaleq E_2 \normaleq \dots \normaleq E_m = E$ of subgroups such that for each $\alpha \in A$, $E_i\alpha = E_i$ for all $0 \le i \le m$. If for all $1\le i\le m$, $Q$ centralizes $E_i/E_{i-1}$, then $Q \le O_p(A)$. In particular, if $E$ is essential and $A=\Aut_{\fs}(E)$, then $Q\le\Inn(E)$.
\end{lemma}
\begin{proof}
This is the fusion theoretic version of \cref{GrpChain}.
\end{proof}

\begin{lemma}\label{E-Bound}
Suppose that $\fs$ is a saturated fusion system and $E$ is an essential subgroup. Assume that $\Aut_{\fs}(E)$ is a $\mathcal{K}$-group. Then $|E/\Phi(E)|\geq |\Out_S(E)|^2$.
\end{lemma}
\begin{proof}
This is \cite[{Lemma 4.5}]{Comp1}.
\end{proof}

Now that we have a way to determine whether a subgroup is essential, in order to make use of the Alperin--Goldschmidt fusion theorem, we must also determine the induced automorphism group by $\fs$. The first result along these lines determines the potential automizer $\Aut_{\fs}(E)$ of an essential subgroup $E$ whenever some non-central chief factor of $E$ is an FF-module. It is important to note that this theorem does not rely on a $\mathcal{K}$-group hypothesis.

\begin{theorem}\label{SEFF}
Suppose that $E$ is an essential subgroup of a saturated fusion system $\fs$ over a $p$-group $S$, and assume that there is an $\Aut_{\fs}(E)$-invariant subgroup $V\le \Omega(Z(E))$ such that $V$ is an FF-module for $G:=\Out_{\fs}(E)$. Then, writing $L:=O^{p'}(G)$, we have that $L/C_L(V)\cong \SL_2(p^n)$, $C_L(V)$ is a $p'$-group and $V/C_V(O^p(L))$ is a natural $\SL_2(p^n)$-module.
\end{theorem}
\begin{proof}
This is \cite[Theorem 1.1]{henkesl2}.
\end{proof}

In this work, we encounter potential essentials subgroups none of whose non-central chief factors are FF-modules for its outer automorphism group. In this case, we still need to determine the automizers and must use other techniques. To do this, we appeal to a proposition proved in \cite{MainThm}. There, this proposition was proved under a $\mathcal{K}$-group hypothesis. However, when we have need of this proposition later in the paper, we will almost always have a workaround which allows us to apply this proposition without the need for this extra hypothesis. To ease exposition, we present the proposition as it is stated in \cite{MainThm} and describe how we avoid the need for a $\mathcal{K}$-group hypothesis in each application as it appears.

\begin{definition}
Suppose that $\fs$ is a saturated fusion system on a $p$-group $S$. Then $E\le S$ is \emph{maximally essential} in $\fs$ if $E$ is essential and, if $F\le S$ is essential in $\fs$ and $E\le F$, then $E=F$.
\end{definition}

\begin{proposition}\label{MaxEssen}
Suppose that $\fs$ is a saturated fusion system on a $p$-group $S$ and $E\le S$ is maximally essential. If $m_p(\Out_S(E))\geq 2$ then $O^{p'}(\Out_{\fs}(E))$ is a $p'$-central extension of one of the following groups:
\begin{enumerate}
\item $\PSL_2(p^a)$ or $\PSU_3(p^b)$ for $p$ arbitrary, $a\geq 2$ and $p^b>2$;
\item $\Sz(2^{2a+1})$ for $p=2$ and $a\geq 1$;
\item $\PSL_2(8)\cong\Ree(3)'$, $\mathrm{Ree}(3^{2a+1}), \PSL_3(4)$ or $\mathrm{M}_{11}$ for $p=3$ and $a\geq 0$;
\item $\Sz(32)$, $\Sz(32):5, {}^2\mathrm{F}_4(2)'$ or  $\mathrm{McL}$ for $p=5$; or
\item $\mathrm{J}_4$ for $p=11$.
\end{enumerate}
\end{proposition}
\begin{proof}
See \cite[Proposition 3.2.6]{MainThm}.
\end{proof}

Now equipped with a method to ascertain a complete set of essentials of $\fs$ along with their automizers, we can apply the Alperin--Goldschmidt fusion theorem. We make use of the main result in \cite{MainThm} wherein fusion systems in which the Alperin-Goldschmidt theorem provides a rank $2$ amalgam are classified.

\begin{theorem}\label{MainThm}
Let $\fs$ be a saturated fusion system with two $\Aut_{\fs}(S)$-invariant maximally essential subgroups $E_1, E_2\normaleq S$ such that $\Aut_{\fs}(E_i)$ is a $\mathcal{K}$-group for $i\in\{1,2\}$. Suppose that $O_p(\fs)=\{1\}$ and $\fs=\langle N_{\fs}(E_1), N_{\fs}(E_2) \rangle$. Then $\fs$ is known explicitly.
\end{theorem}

Rather than provide a full list of outcomes in the above theorem, we instead only list the relevant fusion systems in the following two corollaries.

\begin{corollary}\label{G2Cor}
Suppose the hypothesis of \cref{MainThm} and assume that $S$ is isomorphic to a Sylow $p$-subgroup of $\mathrm{G}_2(p^n)$ for some $n\in\N$. Then either
\begin{enumerate}
\item $\fs=\fs_S(G)$, where $F^*(G)=O^{p'}(G)\cong\mathrm{G}_2(p^n)$;
\item $p=2$ and $\fs=\fs_S(G)$ where $G\cong\mathrm{M}_{12}$ or $\mathrm{G}_2(3)$; or
\item $p=7$, $\fs$ is a uniquely determined simple fusion system on a Sylow $7$-subgroup of $\mathrm{G}_2(7)$ and, assuming the classification of finite simple groups, $\fs$ is exotic.
\end{enumerate}
\end{corollary}

\begin{corollary}\label{PSUCor}
Suppose the hypothesis of \cref{MainThm} and assume that $S$ is isomorphic to a Sylow $p$-subgroup of $\PSU_4(p^n)$ for some $n\in\N$.  Then $\fs=\fs_S(G)$, where $F^*(G)=O^{p'}(G)\cong\mathrm{PSU}_4(p^n)$; or $p^n=3$ and $G\cong\PSU_6(2), \PSU_6(2).2$, $\mathrm{McL}$, $\Aut(\mathrm{McL})$ or $\mathrm{Co}_2$.
\end{corollary}

\section[Fusion Systems on a Sylow \texorpdfstring{$2$}{2}-subgroup of \texorpdfstring{$\mathrm{G}_2(2^n)$}{G2(2n)}]{Fusion Systems on a Sylow $2$-subgroup of $\mathrm{G}_2(2^n)$}

In this section, we let $q=2^n$ for $n\in\N$, $\mathbb{K}=\GF(q)$ and $S$ be isomorphic to a Sylow $2$-subgroup of $\mathrm{G}_2(q)$. Assume throughout that $\fs$ is a saturated fusion system on $S$.

We deal with the $q=2$ case separately in order to streamline some of the arguments later in this section. Fortunately, since $|S|=2^6$ is small, we can directly determine the list of $S$-centric, $S$-radical subgroups and their automizers. We employ MAGMA to do this, although remark that lemmas and propositions in the remainder of this section all apply when $q=2$ and their proofs could adapted with minor alternations.

\begin{proposition}
Let $S$ be isomorphic to a Sylow $2$-subgroup of $\mathrm{G}_2(2)$. The $S$-centric, $S$-radical subgroups of $S$ are $S, C_S(Z_3(S)/Z(S)), C_S(Z_2(S))$ and the maximal elementary abelian subgroups of $S$ of order $2^3$.
\end{proposition}

\begin{proposition}\label{q=2Class}
Let $\fs$ be a saturated fusion system over a Sylow $2$-subgroup of $\mathrm{G}_2(2)$. Set $Q_1:=C_S(Z_3(S)/Z(S))$ and $Q_2=C_S(Z_2(S))$. Then one of the following holds:
\begin{enumerate}
\item $\fs=\fs_S(S)$;
\item $\fs=\fs_S(Q_1: \Out_{\fs}(Q_1))$ where $\Out_{\fs}(Q_1)$ is isomorphic to a subgroup of $(3\times 3):2$;
\item $\fs=\fs_S(Q_2: \Out_{\fs}(Q_2))$ where $\Out_{\fs}(Q_2)\cong \Sym(3)$;
\item $\fs=\fs_S(M)$ where $M\cong 2^3.\PSL_3(2)$;
\item $\fs=\fs_S(G)$ where $G\cong \mathrm{G}_2(2)$;
\item $\fs=\fs_S(G)$ where $G\cong \mathrm{G}_2(3)$; or
\item $\fs=\fs_S(G)$ where $G\cong\mathrm{M}_{12}$.
\end{enumerate}
\end{proposition}

\begin{remark}
In case (iv) of the above theorem, one can take $M$ to be a maximal subgroup of $\mathrm{G}_2(3)$.
\end{remark}

We continue the analysis when $p=2$ and suppose throughout the remainder of this section that $q>2$. We may reduce the commutator formulas from \cref{G2Sylow} to the following:

\begin{align*}
[x_\alpha(t), x_\beta(u)]&=x_{\alpha+\beta}(tu)x_{2\alpha+\beta}(t^2u)x_{3\alpha+\beta}(t^3u)\\
[x_\alpha(t), x_{\alpha+\beta}(u)]&=x_{3\alpha+\beta}(t^2u)x_{3\alpha+2\beta}(tu^2)\\
[x_\alpha(t), x_{2\alpha+\beta}(u)]&=x_{3\alpha+\beta}(tu)\\
[x_\beta(t), x_{3\alpha+\beta}(u)]&=x_{3\alpha+2\beta}(tu)\\
[x_{\alpha+\beta}(t), x_{2\alpha+\beta}(u)]&=x_{3\alpha+2\beta}(tu).
\end{align*}

It follows that 
\[Z_3(S)=\langle X_{\alpha+\beta}, X_{2\alpha+\beta}, X_{3\alpha+\beta}, X_{3\alpha+2\beta}\rangle\] \[Z_2(S)=\langle X_{3\alpha+\beta}, X_{3\alpha+2\beta}\rangle\]\[Z(S)=\langle X_{3\alpha+2\beta}\rangle\] are characteristic subgroups of $S$ of orders $q^4$, $q^2$ and $q$ respectively.

We define \[Q_1:=C_S(Z_3(S)/Z_1(S))=\langle X_\beta, X_{\alpha+\beta}, X_{2\alpha+\beta}, X_{3\alpha+\beta}, X_{3\alpha+2\beta}\rangle\]
\[Q_2:=C_S(Z_2(S))=\langle X_\alpha, X_{\alpha+\beta}, X_{2\alpha+\beta}, X_{3\alpha+\beta}, X_{3\alpha+2\beta}\rangle\]
both of order $q^5$ and characteristic in $S$. Moreover, we can identify $Q_1$ and $Q_2$ with unipotent radicals of two maximal parabolic subgroups in $\mathrm{G}_2(q)$. Therefore, $\Phi(Q_1)=Z(Q_1)=Z(S)$ and $\Phi(Q_2)=Z_2(S)=Z(Q_2)$. 

The following lemma gives detailed information on involutions in $S$, their normalizers and the maximal elementary abelian subgroups of $S$.

\begin{lemma}\label{thomas}
Every involution in $S$ is conjugate in $S$ to one of the following: $x_\alpha(t_1)$,  $x_\beta(t_2)x_{2\alpha+\beta}(t_2')$,  $x_{2\alpha+\beta}(t_3)$,  $x_{\alpha+\beta}(t_4)$,  $x_{3\alpha+\beta}(t_5)$ or  $x_{3\alpha+2\beta}(t_6)$, for $t_i\in\mathbb{K}^\times$ and $t_2'\in\mathbb{K}$. Moreover, each has centralizer of order $q^3$, $q^4$, $q^4$, $q^4$, $q^5$ or $q^6$ respectively. As a consequence, every maximal elementary abelian subgroup is conjugate in $S$ to one of \[T:=X_\alpha X_{3\alpha+\beta}X_{3\alpha+2\beta},\] \[U:=X_\beta X_{2\alpha+\beta} X_{3\alpha+2\beta},\] \[V:=X_\beta X_{\alpha+\beta} X_{3\alpha+2\beta},\] \[W:=X_{2\alpha+\beta} X_{3\alpha+\beta}X_{3\alpha+2\beta}, or\] \[X:=X_{\alpha+\beta} X_{3\alpha+\beta}X_{3\alpha+2\beta}.\] All are of order $q^3$ and have normalizers in $S$ equal to $Q_2$, $Q_1$, $Q_1$, $S$ and $S$ respectively.
\end{lemma}
\begin{proof}
See \cite[(3.6)-(3.10)]{thomas}.
\end{proof}

Throughout this section, we retain the notation from the lemma and remark that $WX=Z_3(S)$, $T\le Q_2$, $T\not\le Q_1$, $U, V\le Q_1$ and $U,V\not\le Q_2$.

We can now begin to determine to the possible essential subgroups of $\fs$. The primary technique used is \cref{Chain} which, more generally, aids in proving that a candidate subgroup $E$ is not an $\fs$-radical subgroup of $S$. Moreover, if we can prove that a chain of characteristic subgroups of $E$ is centralized by some $p$-group not contained in $E$, then $E$ will be not be $S$-radical. For large parts of this section, we can operate in this more general setting, assuming only that $E$ is $S$-centric and $S$-radical.

\begin{proposition}\label{2Qi}
Let $E$ be an $S$-centric and $S$-radical subgroup of $S$ and suppose $Z_3(S)\le E$. Then $E\in\{Q_1, Q_2, S\}$.
\end{proposition}
\begin{proof}
Since $Z_3(S)\le E$, $W,X\le E$ and so $\mathcal{A}(E)\subseteq \mathcal{A}(S)$. Suppose first that $Q_i<E$ for some $i\in\{1,2\}$. Then, $W,X$ are the unique normal elementary abelian subgroups of maximal rank in $E$ and so $Z_3(S)=WX$ is characteristic in $E$. Hence, $Z_2(S)=Z(Z_3(S))$ is also a characteristic subgroup. If $Q_1\not\le E$ and $Q_2\not\le E$, then $\mathcal{A}(E)=\{W,X\}$, $J(E)=Z_3(S)$ and again, $Z_3(S)$ and $Z_2(S)$ are characteristic subgroups of $E$. Thus, we have shown in either case that $Z_2(S)$ and $Z_3(S)$ are characteristic subgroups of $E$.

Now, if $Q_2\not\le E$, $Q_2$ centralizes the chain $\{1\}\normaleq Z_2(S)\normaleq Z_3(S)\normaleq E$ and $E$ is not $S$-radical by \cref{Chain}, a contradiction. So we may as well assume that $Q_2< E$. But then, it follows from the commutator formulas that $Z(E)=Z(S)$. Hence, $Q_1$ centralizes the chain $\{1\}\normaleq Z(S)\normaleq Z_2(S)\normaleq Z_3(S)\normaleq E$, and since $E$ is $S$-radical, we conclude that $E=S$, as required.
\end{proof}

\begin{lemma}\label{Center}
Let $E$ be an $S$-centric, $S$-radical subgroup of $S$ and suppose that $Z_3(S)\not\le E$. Then $Z(S)<Z(E)$ and if $Z(S)<Z(E)\cap Z_2(S)$, then $Z_2(S)<Z(E)$ and $E<Q_2$. In particular $Z(E)\not\le Z_2(S)$. 
\end{lemma}
\begin{proof}
Suppose first that $Z(S)=Z(E)$. Since $WX=Z_3(S)\not\le E$, there exists $Y\in\{W,X\}$ with $Y\not\le E$. Notice that $Z_2(S)$ centralizes the chain $\{1\}\normaleq Z(E)\normaleq E$ so that, as $E$ is $S$-radical, $Z_2(S)\le E$ and $Z_2(S)\le Z_2(E)$. Suppose that $\Omega(Z_2(E))\le Q_1$. Then, as $Y\normaleq S$, $Y$ centralizes the chain $\{1\}\normaleq Z(E)\normaleq \Omega(Z_2(E))\normaleq E$, a contradiction since $Y\not\le E$. Therefore, by \cref{thomas}, there exists an involution $e\in Z_2(E)$ which is conjugate in $S$ to $x_\alpha(t)$, for some $t\in\mathbb{K}^\times$. Since $[E, e]\le Z(E)=Z(S)$ it follows from the commutator formulas that elements of $E$ are conjugate to elements of $Q_2$, and since $Q_2\normaleq S$ we deduce that $E\le Q_2$. But then $Z(S)<Z_2(S)\le Z(E)$, a contradiction. Hence, $Z(S)<Z(E)$.

Suppose now that $Z(S)<Z(E)\cap Z_2(S)$ and let $e\in (Z(E)\cap Z_2(S))\setminus Z(S)$. Then $C_S(e)=Q_2$ by \cref{thomas} and $E\le C_S(e)=Q_2$. Because $E$ is $S$-centric, $Z_2(S)\le E$ from which it follows that $Z_2(S)\le Z(E)$. Assume that $Z(E)=Z_2(S)$. Then, $Q_2$ centralizes the chain $\{1\}\normaleq Z(E)\normaleq E$ and since $E$ is $S$-radical, $Q_2\le E$. But then $Z_3(S)\le E$, a contradiction. Hence, if $Z(S)<Z(E)\cap Z_2(S)$ we deduce that $Z(E)>Z_2(S)$ and $E<Q_2$.
\end {proof}

\begin{proposition}
Let $E$ be an  $S$-centric, $S$-radical subgroup of $S$ and suppose that $Z_3(S)\not\le E$. Then $E$ is maximal elementary abelian, so is conjugate in $S$ to $W,X,T,U$ or $V$.
\end{proposition}
\begin{proof}
By \cref{Center}, we may assume that $Z(E)\not\le Z_2(S)$. Suppose first that $\Omega(Z(E))\le Z_2(S)$. By \cref{Center}, either $\Omega(Z(E))=Z(S)$; or that $Z_2(S)<Z(E)$ and $E<Q_2$. Suppose the latter and, since $Z_3(S)\not\le E$, choose $Y\in\{W,X\}$ with $Y\not\le E$. Since $\Omega(Z(E))\le Z_2(S)<Z(E)$, $E$ is centric and $Z_2(S)$ has exponent $2$, we have that $\Omega(Z(E))=Z_2(S)$ and $Y$ centralizes the chain, $\{1\}\normaleq \Omega(Z(E))\normaleq E$, a contradiction since $E$ is $S$-radical and $Y\not\le E$. Hence, we assume that $\Omega(Z(E))=Z(S)=Z(E)\cap Z_2(S)$ and $E\not\le Q_2$.

Since $Z_2(S)$ centralizes the chain $\{1\}\normaleq \Omega(Z(E))\normaleq E$, $Z_2(S)\le E$ and $Z(E)\le Q_2$. Furthermore, $[Z_3(S), E]\le Z_2(S)\le E$ and so $Z_3(S)\le N_S(E)\le N_S(Z(E))$. In particular, $[Z_3(S), Z(E)]\le Z(E)\cap [Z_3(S), Q_2]=Z(E)\cap Z_2(S)=\Omega(Z(E))=Z(S)$ and so $Z(E)\le C_S(Z_3(S)/Z(S))=Q_1$. Therefore, $Z(E)\le Z_3(S)$. Let $e\in E$ be an involution and suppose that $e\not\le Q_1$. Then, by \cref{thomas}, $e$ is conjugate in $S$ to $x_{\alpha}(t)$ for some $t\in\mathbb{K}^\times$. Thus, $Z(E)\le C_S(e)\le T^s$ for some $s\in S$ and since $Z(E)\le Z_3(S)\normaleq S$, it follows that $Z(E)\le X_{3\alpha+\beta}X_{3\alpha+2\beta}=Z_2(S)$. But then $Z(E)$ has exponent $2$ and $Z(E)=\Omega(Z(E))=Z(S)$, a contradiction. Therefore, $\Omega(E)\le E\cap Q_1$. In particular, $Z_2(S)\le \Omega(E)$ so that $[E, Z_3(S)]\le \Omega(E)$ and $Z_3(S)$ centralizes the chain $\{1\}\normaleq \Omega(Z(E))\normaleq \Omega(E)\normaleq E$, a contradiction since $E$ is $S$-radical and $Z_3(S)\not\le E$. Therefore, $\Omega(Z(E))\not\le Z_2(S)$.

Hence, there exists an involution $e\in Z(E)\setminus Z_2(S)$ such that $e$ is conjugate in $S$ to $x_\alpha(t_1)$,  $x_\beta(t_2)x_{2\alpha+\beta}(t_2')$,  $x_{2\alpha+\beta}(t_3)$ or $x_{\alpha+\beta}(t_4)$ for $t_i\in\mathbb{K}^\times$ and $t_2'\in\mathbb{K}$ by \cref{thomas}. Suppose first that $e$ is conjugate to $x_{\alpha}(t)$, some $t\in\mathbb{K}^\times$. Then $E\le C_S(e)=T^s$ for some $s\in S$ and since $E$ is $S$-centric, $E=T^s$. 

Suppose now that $e$ is conjugate to $x_{2\alpha+\beta}(t)$, $t\in\mathbb{K}^\times$. Then $E\le C_S(e)=WU^s\le Q_1$ for some $s\in S$ and $Z(C_S(e))=(U\cap W)^s\le Z(E)$. If $Z(C_S(e))=Z(E)$, then $C_S(e)$ centralizes the series $\{1\}\normaleq Z(E)\normaleq E$ and $E=C_S(e)$. But now, $X$ centralizes the series $\{1\}\normaleq E'\normaleq E$ and since $E$ is $S$-radical and $X\not\le E$, we have a contradiction. Thus, $Z(C_S(e))<\Omega(Z(E))$ and $C_S(\Omega(Z(E)))$ is an elementary abelian subgroup of order $q^3$. Since $E$ is $S$-centric, it follows that $|E|=q^3$ and $E=W$ or $U^s$ for some $s\in S$, as required. If $e$ is conjugate to $x_{\alpha+\beta}(t)$, we obtain $E\le C_S(e)=XV^s$ for some $s\in S$ by \cref{thomas}. Arguing as before, we conclude that $E$ is conjugate to either $V$ or $X$ in $S$.

Finally, we suppose that $e$ is conjugate in $S$ to some $x_{\beta}(t)x_{2\alpha+\beta}(t')$, for $t\in\mathbb{K}^\times$ and $t'\in\mathbb{K}$. Then, using the commutator formulas, one can calculate that $|C_S(e)|=q^4$, $E\le C_S(e)\le Q_1$ and $Z(S)X_{\beta}^s=Z(C_S(e))\le \Omega(Z(E))$ for some $s\in S$. If $\Omega(Z(E))=Z(C_S(e))$ then $C_S(e)$ centralizes the series $\{1\}\normaleq Z(E) \normaleq E$ and since $E$ is $S$-radical, $E=C_S(e)$. But then, $E'=Z(S)$ and $Q_1$ centralizes the series $\{1\}\normaleq E'\normaleq E$, a contradiction as $E$ is $S$-radical and $Q_1\not\le E$. Hence, $Z(C_S(e))<\Omega(Z(E))$, $|\Omega(Z(E))|>q^2$ and as $\Omega(Z(E))Z_3(S)\le Q_1$, there is some $\wt e\in (\Omega(Z(E))\cap Z_3(S))\setminus Z(S)$. Indeed, $\wt e$ is not contained in $Z_2(S)$, for otherwise $E\le Q_1\cap Q_2=Z_3(S)$. Therefore, $\wt e$ is conjugate in $S$ to some $x_{2\alpha+\beta}(t)$ or $x_{\alpha+\beta}(t)$ and by the above, $E$ is elementary abelian. Moreover, since there is $e\in E$ conjugate to some $x_{\beta}(t)x_{2\alpha+\beta}(t')$, we have that $E$ is conjugate to $U$ or $V$.
\end{proof}

We have shown that the $S$-centric, $S$-radical subgroups of $S$ are $S$, $Q_1$, $Q_2$ and the maximal elementary abelian subgroups of $S$. At this point, we restrict our attention to a saturated fusion system $\fs$ on $S$ and its essential subgroups. We make use of \cref{E-Bound}, and as stated, this appears to rely on a $\mathcal{K}$-group hypothesis on $\Aut_{\fs}(E)$, where $E$ is a candidate essential subgroup. Following the proof in \cite[{Proposition 4.8}]{Comp1}, the $\mathcal{K}$-group condition is only used to provide a list of candidates for groups with a strongly $2$-embedded subgroup along with their Sylow $2$-subgroups. Fortunately, when $p=2$ a result of Bender \cite{Bender} classifies all such groups and so, we can determine the essential subgroups of $\fs$ without the need to employ a $\mathcal{K}$-group hypothesis.

In addition, the proof of \cref{MaxEssen} relies on a $\mathcal{K}$-group hypothesis for the same reason as \cref{E-Bound} and so when $p=2$, utilizing Bender's result with the acknowledgment that $q>2$, $O^{2'}(\Out_{\fs}(E))$ is isomorphic to a central extension of a rank $1$ group of Lie type in characteristic $2$, independent of any $\mathcal{K}$-group hypothesis on $\Aut_{\fs}(E)$. A final consideration is that we intend to use \cref{G2Cor} which relies on the \cref{MainThm} which again uses a $\mathcal{K}$-group hypothesis. Following the proof of that theorem, the determination of $\fs$ from a rank $2$ amalgam relies only on the work in \cite{Greenbook} which is, again, independent of any $\mathcal{K}$-group hypothesis. Hence, when $p=2$, we can apply all the necessary results to determine $\fs$ without the need to enforce a $\mathcal{K}$-group hypothesis on $\Aut_{\fs}(E)$.

\begin{theorem}
Let $\fs$ be a saturated fusion system over a Sylow $2$-subgroup of $\mathrm{G}_2(2^n)$ for $n>1$. Then one of the following holds:
\begin{enumerate}
\item $\fs=\fs_S(S: \Out_{\fs}(S))$;
\item $\fs=\fs_S(Q_i: \Out_{\fs}(Q_i))$ where $O^{2'}(\Out_{\fs}(Q_i))\cong \SL_2(2^n)$; or 
\item $\fs=\fs_S(G)$, where $F^*(G)=O^{2'}(G)\cong\mathrm{G}_2(2^n)$.
\end{enumerate}
\end{theorem}
\begin{proof}
Let $E\in \mathcal{E}(\fs)$ and suppose that $E$ is elementary abelian. Then, in all cases, we deduce that $q^3=|E|<q^4\leq |\Out_S(E)|^2$, a contradiction by \cref{E-Bound}. Therefore, $\mathcal{E}(\fs)\subseteq \{Q_1, Q_2\}$. If neither $Q_1$ nor $Q_2$ are essential then outcome (i) holds, and if $\mathcal{E}(\fs)=\{Q_i\}$ for some $i\in\{1,2\}$ then since $Q_i$ is $\Aut_{\fs}(S)$-invariant, outcome (ii) holds by \cref{MaxEssen}. Thus, $\mathcal{E}(\fs)=\{Q_1, Q_2\}$. Since $Q_i$ is $\Aut_{\fs}(S)$-invariant for $i\in\{1,2\}$, if $O_2(\fs)=\{1\}$ we apply \cref{G2Cor} and the result follows.

Suppose that $Q:=O_2(\fs)\ne\{1\}$. By \cref{normalinF}, $Q\le Q_1\cap Q_2=Z_3(S)$ and so, $\Phi(Q)\le Z(S)$. Now, $Z_2(S)$ is normalized by $\Aut_{\fs}(Q_2)$ and $\Out_S(Q_2)$ centralizes $Z(S)$ which has index $q$ in $Z_2(S)$, which is itself of order $q^2$. Moreover, since $S$ does not centralize $Z_2(S)$, $\Out_S(Q_2)$ acts non-trivially on $Z_2(S)$ and, by \cref{SEFF}, $Z_2(S)$ is an FF-module for $O^{2'}(\Out_{\fs}(Q_2))\cong\SL_2(q)$ and $Z_2(S)$ is irreducible. Since $\Phi(Q)\le Z(S)< Z_2(S)$, we conclude that $\Phi(Q)=\{1\}$, $Q$ is elementary abelian and $Z_2(S)\le Q$. 

If $Q=Z_2(S)$, then $Z_2(S)$ is $\Aut_{\fs}(Q_1)$-invariant and so is $Z_3(S)=C_{Q_1}(Z_2(S))$. But then $S$ centralizes the chain $\{1\}\normaleq Z(S)\normaleq Z_2(S)\normaleq Z_3(S)\normaleq Q_1$, a contradiction since $Q_1$ is $\fs$-radical. Hence, $Z_2(S)<Q<Z_3(S)$ and there is an involution $x\in Q$ which is conjugate in $S$ to $x_{2\alpha+\beta}(t)$ or $x_{\alpha+\beta}(t)$ for some $t\in\mathbb{K}^\times$. But then $C_S(Q)\le Q_1\cap Q_2$ and so $C_S(Q)$ is $\Aut_{\fs}(Q_i)$-invariant for $i\in\{1,2\}$. It follows from \cref{normalinF} that $C_S(Q)\normaleq \fs$ so that $Q=C_S(Q)$ is self-centralizing in $S$, $Q\in\{W,X\}$ and $\fs$ is satisfies the hypothesis of \cref{model}.

By \cref{model}, there is a finite group $G$ such that $F^*(G)=Q$ and $\fs=\fs_S(G)$. Moreover, $O^{2'}(\Out_G(Q_i))\cong \SL_2(q)$ and $\Out_{G}(Q_i)$ acts faithfully on $Q_i/Q$ for $i\in\{1,2\}$. Set $\bar{G}:=G/Q$ and notice that $\bar{Q_1}$ and $\bar{Q_2}$ are self-centralizing in $\bar{G}$. Moreover, $\bar{G}=\langle N_{\bar{G}}(\bar{Q_1}), N_{\bar{G}}(\bar{Q_2})\rangle$, and $\bar{Q_i}$ is $\Aut_{\bar{G}}(\bar{S})$-invariant for $i\in\{1,2\}$. It follows that $\bar{G}$ has a weak BN-pair of rank $2$ in the sense of \cite{Greenbook}. Moreover, since $Q_2$ centralizes $Z_2(S)$ which has index $q$ in $Q$ and $Q_2/Q$ is elementary abelian of order $q^2$, we infer that $Q$ is an FF-module for $\bar{G}$. Then, comparing with the completions in \cite{Greenbook} and applying \cite[Theorem A]{ChermakJ}, $Q$ is a ``natural module'' for $O^{2'}(\bar{G})\cong\SL_3(q)$. Notice that if $S$ splits over $Q$, then $S$ is isomorphic to a Sylow $2$-subgroup of $\PSL_4(q)$. Then by \cite[Theorem 3.3.3]{GLS3}, the $2$-rank of $S$ is $4n$, a contradiction to \cref{thomas}. Therefore, $S$ is non-split and it follows from \cite[Table I]{bell} that $q=2$, a contradiction to the original hypothesis.
\end{proof}

Combined with the classification provided in \cref{q=2Class}, this completely determines all saturated fusion systems on a Sylow $2$-subgroup of $\mathrm{G}_2(2^n)$ for any $n\in\N$.

\section[Fusion Systems on a Sylow \texorpdfstring{$3$}{3}-subgroup of \texorpdfstring{$\mathrm{G}_2(3^n)$}{G2(3n)}]{Fusion Systems on a Sylow $3$-subgroup of $\mathrm{G}_2(3^n)$}

Throughout this section, we suppose that $p=3$, $q=3^n$ for $n\in\N$, $\mathbb{K}$ is a finite field of order $q$ and $S$ is isomorphic to a Sylow $3$-subgroup of $\mathrm{G}_2(q)$. We may reduce the commutator formulas from \cref{G2Sylow} to the following:

\begin{align*}
[x_\alpha(t), x_\beta(u)]&=x_{\alpha+\beta}(-tu)x_{2\alpha+\beta}(-t^2u)x_{3\alpha+\beta}(t^3u)x_{3\alpha+2\beta}(t^3u^2)\\
[x_\alpha(t), x_{\alpha+\beta}(u)]&=x_{2\alpha+\beta}(tu)\\
[x_\beta(t), x_{3\alpha+\beta}(u)]&=x_{3\alpha+2\beta}(tu).
\end{align*}

Additionally, we remark that $Z(S)=\langle X_{2\alpha+\beta}, X_{3\alpha+2\beta}\rangle$ is a characteristic subgroup of $S$ of order $q^2$.

We let \[Q_1=\langle X_\beta,  X_{3\alpha+\beta}, X_{\alpha+\beta}, X_{2\alpha+\beta}, X_{3\alpha+2\beta}\rangle\] \[Q_2=\langle X_\alpha, X_{\alpha+\beta}, X_{3\alpha+\beta}, X_{3\alpha+2\beta}, X_{2\alpha+\beta}\rangle\] and by removing one root subgroup at a time from $Q_i$, starting from the left, we get a chain of subgroups $Q_1\cap Q_2\to Z(Q_i)\to Z(S) \to\Phi(Q_i)\to \{1\}$ e.g. \[Z(Q_1)=\langle X_{\alpha+\beta}, X_{2\alpha+\beta}, X_{3\alpha+2\beta}\rangle.\]

Before determining the essential subgroups of a saturated fusion system $\fs$ on $S$, we state and prove some important properties of $S, Q_1$ and $Q_2$ which may be of interest in their own right.

\begin{lemma}\label{SL3Sub}
The subgroup $X:=\langle X_{\beta}, X_{3\alpha+\beta}, X_{3\alpha+2\beta}\rangle\le Q_1$ is of shape $q^{1+2}$ and is isomorphic to a Sylow $3$-subgroup of $\SL_3(q)$.
\end{lemma}
\begin{proof}
Since the groups $X_\beta$ and $X_{3\alpha+\beta}$ commute modulo $X_{3\alpha+2\beta}$, it follows that every element may be written as $x_{3\alpha+\beta}(t_1)x_\beta(t_2)x_{3\alpha+2\beta}(t_3)$ for $t_1,t_2, t_3\in\mathbb{K}$. Then, using the commutator formulas, we calculate that the map $\theta:X\to\SL_3(q)$ such that 
\[
(x_{3\alpha+\beta}(t_1)x_\beta(t_2)x_{3\alpha+2\beta}(t_3))\theta=
\begin{pmatrix} 
1 & 0 & 0\\ 
t_1 & 1 & 0 \\ 
t_3 & t_2 &1
\end{pmatrix}
\]
is an injective homomorphism, from which it follows that $X$ is isomorphic to a Sylow $3$-subgroup of $\SL_3(q)$.
\end{proof}

\begin{remark}
By symmetry, the subgroup $\langle X_\alpha, X_{\alpha+\beta}, X_{2\alpha+\beta}\rangle\le Q_2$ is also isomorphic to a Sylow $3$-subgroup of $\SL_3(q)$.
\end{remark}

As $Q_1=Z(Q_1)X$, we observe that $Q_1$ and $Q_2$ are isomorphic groups of shape $q^2\times q^{1+2}$. We may identify $Q_1, Q_2$ with the radical subgroups of maximal parabolic subgroups of $\mathrm{G}_2(q)$ of shape $(q^2\times q^{1+2}): \GL_2(q)$.

\begin{lemma}\label{SL3Quo}
Let $i\in\{1,2\}$. Then $S/Z(Q_i)$ is isomorphic to a Sylow $3$-subgroup of $\SL_3(q)$.
\end{lemma}
\begin{proof}
Since $X_\alpha Z(Q_1)$ and $X_\beta Z(Q_1)$ commute modulo $X_{3\alpha+\beta}Z(Q_1)/Z(Q_1)$ we may write any element of $S/Z(Q_1)$ as $x_{\beta}(t_1)x_{\alpha}(t_2)x_{3\alpha+\beta}(t_3)Z(Q_1)$ for $t_1, t_2, t_3\in\mathbb{K}$. Then the map $\theta_1:S/Z(Q_1)\to\SL_3(q)$ such that 
\[
(x_{\beta}(t_1)x_\alpha(t_2)x_{3\alpha+\beta}(t_3)Z(Q_1))\theta=
\begin{pmatrix} 
1 & 0 & 0\\ 
t_1 & 1 & 0 \\ 
t_3 & t_2^3 &1
\end{pmatrix}
\]
is an injective homomorphism, from which it follows that $S/Z(Q_1)$ is isomorphic to a Sylow $3$-subgroup of $\SL_3(q)$.

In a similar manner, since $X_\alpha Z(Q_2)/Z(Q_2)$ and $X_\beta Z(Q_2)/Z(Q_2)$ commute modulo $X_{\alpha+\beta}Z(Q_2)/Z(Q_2)$ we may write elements of $S/Z(Q_2)$ as $x_{\alpha}(t_1)x_{\beta}(t_2)x_{\alpha+\beta}(t_3)Z(Q_2)$ for $t_1, t_2, t_3\in\mathbb{K}$. Then the map $\theta_2:S/Z(Q_2)\to\SL_3(q)$ such that 
\[
(x_{\alpha}(t_1)x_\beta(t_2)x_{\alpha+\beta}(t_3)Z(Q_2))\theta_2=
\begin{pmatrix} 
1 & 0 & 0\\ 
t_1 & 1 & 0 \\ 
t_3 & t_2 &1
\end{pmatrix}
\]
is an injective homomorphism, from which it follows that $S/Z(Q_2)$ is isomorphic to a Sylow $3$-subgroup of $\SL_3(q)$.
\end{proof}

We summarize some further structural results concerning $S, Q_1$ and $Q_2$. Some are easily calculated using the commutator formulas, while others are lifted from \cite[Definition 2.1]{parkerBN} and \cite[Lemma 6.5]{parkerBN}.

\begin{lemma}\label{pStructure}
For $i\in\{1,2\}$, we have the following:
\begin{enumerate}
\item $Q_1\cap Q_2= Z(Q_1)Z(Q_2)\in\mathcal{A}(S)$ has order $q^4$;
\item $S$ has nilpotency class $3$;
\item $C_S(Z(Q_i))=Q_i$, $|Z(Q_i)|=q^3$, $Z(Q_1)\cap Z(Q_2)=\Phi(Q_1)\times \Phi(Q_2)=Z(S)$ is of order $q^2$ and $\Phi(Q_i)$ is of order $q$;
\item $[Q_i, Z(Q_{3-i})]=\Phi(Q_i)$;
\item for $x\in S\setminus Q_i$ we have that $[x, Q_i]Z(Q_i)=Q_1\cap Q_2$ and $[x, Z(Q_i)]\Phi(Q_i)=Z(S)$;
\item $Q_i$ is of exponent $3$, $S$ is of exponent $9$, $\Omega(S)=S$ and $\mho(S)=Z(S)$;
\item  if $z\in S$ is of order $3$ then $z\in Q_1\cup Q_2$; and
\item if $x\in Q_1\setminus Q_2$ and $y\in Q_2\setminus Q_1$ then $[y,x,x]\ne 1 \ne [x,y,y]$.
\end{enumerate}
\end{lemma}

\begin{lemma}\label{exp3} 
Suppose $R\le S$ has exponent $3$.  Then $R\le Q_1$ or $R\le Q_2$.
\end{lemma}
\begin{proof}
As $R$ has exponent $3$, $R \subset Q_1 \cup Q_2$ by \cref{pStructure} (vii). If $R \not \le Q_1$ and $R \not \le Q_2$,  then there exists $r \in R \setminus Q_1$ and $s \in R \setminus Q_2$.  But then $rs \not \in Q_1 \cup Q_2$, which is impossible.
\end{proof}

\begin{lemma}\label{swapping}
Let $S$ be isomorphic to a Sylow $3$-subgroup of $\mathrm{G}_2(q)$ where $q=3^n$. Then $Q_1\cap Q_2$ is characteristic in $S$, $N_{\Aut(S)}(Q_1)=N_{\Aut(S)}(Q_2)$ has index $2$ in $\Aut(S)$ and for $\alpha\in \Aut(S)$ with non-trivial image in $\Aut(S)/N_{\Aut(S)}(Q_i)$, $Q_i\alpha=Q_{3-i}$ for $i\in\{1,2\}$.
\end{lemma}
\begin{proof}
By \cref{exp3}, $Q_1$ and $Q_2$ are the only subgroups of $S$ of order $q^5$ and exponent $3$. Therefore $\Aut(S)$ permutes $\{Q_1,Q_2\}$. As $Q_1$ and $Q_2$ are exchanged in $\Aut(S)$, $N_{\Aut(S)}(Q_1)$ has index $2$ in $\Aut(S)$ and $N_{\Aut(S)}(Q_1)=N_{\Aut(S)}(Q_2)$. Furthermore, it follows that $Q_1\cap Q_2$ is a characteristic subgroup of $S$.
\end{proof}

\begin{proposition}\label{AutS}
Let $S$ be isomorphic to a Sylow $3$-subgroup of $\mathrm{G}_2(q)$ where $q=3^n$. Then $\Aut(S)=CH$ where $C$ is a normal $3$-subgroup and $H=N_{\Aut(\mathrm{G}_2(q))}(S)$.
\end{proposition}
\begin{proof}
We have that $|N_{\Aut(\mathrm{G}_2(q))}(S)|=q^6.(q-1)^2.2n$ where $q=3^n$, and so $|\Aut(S)|_{3'}\geq (q-1)^2.2n$. Note that $N_{\Aut(S)}(Q_1)=N_{\Aut(S)}(Q_2)$ normalizes $Z(Q_1)$ and $Z(Q_2)$ and so acts on both $S/Z(Q_1)$ and $S/Z(Q_2)$. Let $\alpha\in N_{\Aut(S)}(Q_1)$. If $\alpha$ acts trivially on $S/Z(Q_1)$ and $S/Z(Q_2)$, then $\alpha$ acts trivially on $S/Z(S)$ and since $Z(S)\le \Phi(S)$, $\alpha$ acts trivially on $S/\Phi(S)$.  By \cref{burnside}, all such automorphism form a normal $3$-subgroup of $\Aut(S)$. Now, every other automorphism acts non-trivially on $S/Z(Q_i)$ for some $i\in\{1,2\}$ and so embeds in $\Aut(S/Z(Q_i))$. Without loss of generality, let $i=1$. By \cref{SL3Quo}, $S/Z(Q_1)$ is isomorphic to a Sylow $3$-subgroup of $\SL_3(q)$, and by \cite[Proposition 5.3]{parkerBN}, $\Aut(S/Z(Q_1))=A.\Gamma\mathrm{L}_2(q)$ where $A$ is a normal $3$-subgroup of $\Aut(S/Z(Q_1))$ which centralizes $S/Q_1\cap Q_2$. In particular, setting $C=C_{\Aut(S)}(S/Q_1\cap Q_2)$, $C$ is a normal $3$-subgroup of $\Aut(S)$ and $\Aut(S)/C$ has an index $2$ subgroup which normalizes $Q_1$ and is isomorphic to a subgroup of $\Gamma\mathrm{L}_2(q)$. Specifically, $N_{\Aut(S/Z(Q_1))}(Q_1/Z(Q_1))=N_{\Aut(S/Z(Q_1))}(T)$ where $T\in\syl_3(\Aut(S/Z(Q_1)))$. Therefore, $|\Aut(S)|_{3'}\leq (q-1)^2.2n$ and it follows that $|\Aut(S)|_{3'}=|N_{\Aut(\mathrm{G}_2(q))}(S)|_{3'}$ and $\Aut(S)=CH$ where $C=C_{\Aut(S)}(S/Q_1\cap Q_2)$ and $H=N_{\Aut(\mathrm{G}_2(q))}(S)$.
\end{proof}

\begin{lemma}\label{QiCent}
Let $x\in Q_i\setminus Z(Q_i)$. Then $|C_{Q_i}(x)|=q^4$ and $\mathcal{A}(Q_i)=\{C_{Q_i}(x)\mid x\in Q_i\setminus Z(Q_i)\}$.
\end{lemma}
\begin{proof}
By symmetry, we may as well suppose that $i=1$. Then \cref{SL3Sub} implies that $Q_1= Z(Q_1)X$. Moreover, for $x\in Q_1\setminus Z(Q_1)$, $C_{Q_1}(x)= Z(Q_1)C_X(x)$ and an easy calculation in $X$ shows that $C_X(x)$ has order $q^2$. Hence $C_{Q_1}(x)$ is elementary abelian of order $q^4$. Since the maximal elementary abelian subgroups of $X$ have order $q^2$, the result follows.
\end{proof}

We now determine the set of essential subgroups of a saturated fusion system $\fs$ on $S$ over a series of lemmas and propositions. As in the case where $p=2$, it is enough to assume that a candidate essential is $S$-radical and $S$-centric and so we perform the analysis in this more general setting.

\begin{lemma}\label{Q1/Q2}
Let $E$ be an $S$-centric, $S$-radical subgroup of $S$ and suppose that $Q_1\cap Q_2< E$. Then $Q_1\le E$, $Q_2\le E$ or $E=S$.
\end{lemma}
\begin{proof} 
Suppose that $E$ is an $S$-centric, $S$-radical subgroup with $Q_1\cap Q_2)<E$, $Q_1\not\le E$ and $Q_2\not\le E$. Note that $E\normaleq S$ as $S'\le Q_1\cap Q_2<E$. Since all elements of $S$ of order $3$ are contained in $Q_1\cup Q_2$ we deduce that $\Omega(E)=(Q_1\cap E)(Q_2\cap E)$. Let $\alpha\in\Aut(E)$ and notice that $\Omega(E)$ is characteristic in $E$, so is normalized by $\alpha$. Suppose also that $(Q_1\cap E)\alpha \ne (Q_1\cap E)$. We follow the same argument as \cref{AutS} to see that $(Q_1\cap E)\alpha =(Q_2\cap E)$ and $(Q_2\cap E)\alpha=(Q_1\cap E)$ so that $\alpha$ fixes $(Q_1\cap Q_2 \cap E)$. Therefore, in all cases, at least one of $(Q_1\cap E)$, $(Q_2\cap E)$ or $(Q_1\cap Q_2\cap E)=Q_1\cap Q_2$ is characteristic in $E$.

Suppose $Q_1\cap Q_2$ is characteristic in $E$. If $E< Q_i$ for some $i\in\{1,2\}$, then as $E$ is $S$-centric, $Z(Q_i)\le Z(E)$. If $Z(Q_i)=Z(E)$ then $Q_i$ centralizes the chain $\{1\}\normaleq Z(E)\normaleq E$, a contradiction since $Q_i\not\le E$ and $E$ is $S$-radical. Hence, there is $e\in Z(E)\setminus Z(Q_1)$ and since $Q_1\cap Q_2$ is a maximal elementary abelian subgroup of $S$ which centralizes $Z(E)$, by \cref{QiCent}, we conclude that $E\le C_S(Z(E))=Q_1\cap Q_2$, a contradiction. Therefore, $E\not\le Q_i$ for $i\in\{1,2\}$. We have that $[E,S]\le [S,S]=S'\le Q_1\cap Q_2$ and since $E\not\le Q_i$, we have that $[Q_1\cap Q_2, E]=[Z(Q_1), E][Z(Q_2), E]=Z(S)=[Q_1\cap Q_2, S]$. But $[Q_1\cap Q_2, E]$ is a commutator of two characteristic subgroups of $E$, so is characteristic in $E$. Thus, $S$ centralizes the characteristic chain $\{1\}\normaleq [Q_1\cap Q_2, E] \normaleq Q_1\cap Q_2 \normaleq E$, and since $E$ is $S$-radical, we conclude that $E=S$.

Suppose now that $Q_1\cap E$ is characteristic in $E$ and $Q_1\cap Q_2\le E$ is not characteristic. Then $Q_1\cap Q_2< Q_1\cap E$ and $Z(Q_1 \cap E)$ centralizes $Q_1\cap Q_2$. Since $Q_1\cap Q_2$ is maximal elementary abelian, $Z(Q_1)\le Z(Q_1\cap E)\le Q_1\cap Q_2$. If there is $x\in Z(Q_1\cap E)\setminus Z(Q_1)$ then by \cref{QiCent}, $C_{Q_1}(x)=Q_1\cap Q_2$. But then $Q_1\cap E$ obviously centralizes $x$ so that $Q_1\cap E=Q_1\cap Q_2$  is characteristic in $E$, a contradiction. Therefore, we deduce that $Z(Q_1\cap E)= Z(Q_1)$. But now $[Q_1, E]\le Q_1\cap E$, $[Q_1,Q_1\cap E]\le Q_1'\le Z(Q_1\cap E)$ and $[Q_1, Z(Q_1\cap E)]=\{1\}$ so that $Q_1$ centralizes the chain $\{1\}\normaleq Z(Q_1\cap E)\normaleq Q_1\cap E \normaleq E$ and since $E$ is $S$-radical, $Q_1=Q_1\cap E$ is a characteristic subgroup of $E$. The argument when $Q_2\cap E$ is characteristic in $E$ is similar.
\end{proof}

\begin{proposition}\label{Q_1}
Let $E$ be an $S$-centric, $S$-radical subgroup of $S$ such that $Q_1\cap Q_2<E<S$. Then $E=Q_i$.
\end{proposition}
\begin{proof}
By \cref{Q1/Q2}, we may assume that $Q_1\le E$ or $Q_2\le E$. Without loss of generality, suppose that $Q_1< E$ but $Q_2\not\le E$. By the proof of \cref{Q1/Q2}, $Q_1$ is characteristic in $E$. By the Dedekind modular law, $E=E\cap S=E\cap Q_1Q_2=Q_1(E\cap Q_2)$ so that there exists $x\in (E\cap Q_2) \setminus Q_1$. As a consequence, using the commutator formulas, we deduce that $E'Z(Q_1)=Q_1\cap Q_2$ is a characteristic subgroup of $E$ and $Z(E)=Z(S)$. But then $Q_2$ centralizes the chain $\{1\}\normaleq Z(E)\normaleq Q_1\cap Q_2\normaleq E$, a contradiction since $Q_2\not\le E$ and $E$ is $S$-radical. Therefore, $E=Q_1$, as required.
\end{proof}

\begin{proposition}\label{Q_12}
Let $E\le S$ be an $S$-centric, $S$-radical subgroup of $S$ such that $Q_1\cap Q_2\not\le E$. Then for some $i\in\{1,2\}$, $E\in\mathcal{A}(Q_i)$ is of order $q^4$ and $N_S(E)=Q_i$.
\end{proposition}
\begin{proof}
Suppose that $Q_1\cap Q_2\not\le E$. If $Z(E)\le Q_1\cap Q_2$, since $[E, Q_1\cap Q_2]\le [S, Q_1\cap Q_2]= Z(S)\le Z(E)$, $Q_1\cap Q_2$ centralizes the chain $\{1\}\normaleq Z(E)\normaleq E$, a contradiction since $E$ is $S$-radical. Thus, $Z(E)\not\le Q_1\cap Q_2$. Since $Q_1\cap Q_2\not \le E$, and $Q_1\cap Q_2=Z(Q_1)Z(Q_2)$, we may assume without loss of generality that $Z(Q_1)\not\le E$. If $\Omega(Z(E))\le Q_1$ then, since $[E, Z(Q_1)]\le [S, Z(Q_1)]=Z(S)\le \Omega(Z(E))$, $Z(Q_1)$ centralizes the chain $\{1\}\normaleq \Omega(Z(E)) \normaleq E$, a contradiction.

Hence, $\Omega(Z(E))\not\le Q_1$ and so, $\Omega(Z(E))\le Q_2$ by \cref{exp3}. Since $E$ centralizes $\Omega(Z(E))$, it follows from the commutator formulas that $E\le Q_2$ and since $E$ is $S$-centric, we conclude $Z(Q_2)\le \Omega(Z(E))$. Moreover, since $Z(E)\not\le Q_1\cap Q_2$, there exists $e\in Z(E) \setminus Z(Q_2)$ and therefore $E \le C_S(e)\in\mathcal{A}(Q_2)$ by \cref{SL3Sub}. Since $E$ is $S$-centric, $E=C_S(e)$ is elementary abelian of order $q^4$ and calculating using the commutator formulas, it follows that $N_S(E)=Q_2$. A similar argument when $Z(Q_2)\not\le E$ completes the proof.
\end{proof}

Having identified the $S$-centric, $S$-radical subgroups we now turn our attention to a fixed saturated fusion system $\fs$ on $S$ and its essential subgroups. In the following, to restrict the list of centric, radical subgroups, we make use of \cref{SEFF}, again stressing that this theorem does not rely on $\mathcal{K}$-group hypothesis. Moreover, we use some results in \cite{parkersem} and even though the hypothesis there includes $O_3(\fs)=\{1\}$, the results we use are independent of this. Thus, we can still operate in a completely general setting. 

\begin{lemma}\label{Q1/Q22}
Let $E$ be an essential subgroup of a saturated fusion system $\fs$ on $S$. Then $Q_1\cap Q_2\le E$.
\end{lemma}
\begin{proof}
By \cref{Q_12}, without loss of generality, we assume that $E$ is a maximal elementary abelian subgroup of $N_S(E)=Q_2$, $E\cap Q_1=Z(Q_2)$ and $E(Q_1\cap Q_2)=Q_2$. Since $Z(Q_2)$ is an index $q$ subgroup of $E$ centralized by $Q_2$, it follows by \cref{SEFF} that $O^{3'}(\Out_{\fs}(E))\cong \SL_2(q)$ and $E/ C_E(O^{3'}(\Out_{\fs}(E)))$ is a natural $\SL_2(q)$-module. Set $Z_E:=C_E(O^{3'}(\Out_{\fs}(E)))\le Z(Q_2)$ and let $t_E \in Z(O^{3'}(\Out_{\fs}(E)))$. By \cref{Q_1} and \cref{Q_12}, every essential subgroup is contained in either $Q_1$ or $Q_2$. In particular, we may as well assume that the only possible essential subgroup $E$ is strictly contained in is $Q_2$. Since $t_E$ normalizes $\Out_S(E)$, using that $E$ is receptive, and applying the Alperin--Goldschmidt theorem, we conclude that $t_E$ lifts to some automorphism of $S$ or $Q_2$, and since $Q_2=N_S(E)$, the lift of $t_E$ normalizes $Q_2$ in both cases. 

Suppose that $t_E$ lifts to some automorphism of $S$ and call this morphism $t_E^*$. Since $t_E^*$ normalizes $Q_2$, by \cref{swapping}, $t_E^*$ normalizes $Q_1$. Moreover, $t_E^*$ centralizes $Z(Q_1)/Z(S)=Z(Q_1)/(Z(Q_1)\cap E)\cong Q_2/E$. Since $t_E^*$ normalizes $\Phi(Q_2)$, either $t_E^*$ inverts $\Phi(Q_2)$ or centralizes $\Phi(Q_2)$. If $t_E^*$ centralizes $\Phi(Q_2)$, then $[Q_1\cap Q_2, Q_2, t_E^*]=\{1\}$. But $t_E^*$ centralizes $(Q_1\cap Q_2)/Z(Q_2)=(Q_1\cap Q_2)/(Q_1\cap Q_2\cap E)\cong Q_2/E$ so that $[t_E^*, Q_1\cap Q_2, Q_2]=\{1\}$. Then, the three subgroup lemma yields $[t_E^*, Q_2, Q_1\cap Q_2]=\{1\}$ so that $[t_E^*, Q_2]\le E\cap Q_1\cap Q_2=Z(Q_2)$, a contradiction since $Z_E\le Z(Q_2)$. Thus, $t_E^*$ inverts $\Phi(Q_2)$ and since $Z_E\le Q_2$ has order $q^2$, it follows that $t_E^*$ centralizes $Z(Q_2)/\Phi(Q_2)$ and $(Q_1\cap Q_2)/\Phi(Q_2)=C_{Q_2/\Phi(Q_2)}(t_E^*)$. Again, $t_E^*$ either inverts $S/Q_2$ or centralizes $S/Q_2$. Suppose the latter. Then $t_E^*Q_2$ is normalized by $S$ so that $[Q_2/\Phi(Q_2), t_E^*]$ is normalized by $S$. But $Z(S/\Phi(Q_2))\le (Q_1\cap Q_2)/\Phi(Q_2)=C_{Q_2/\Phi(Q_2)}(t_E^*)$ from which it follows that $[Q_2/\Phi(Q_2), t_E^*]=\{1\}$, a clear contradiction. Thus, $t_E^*$ inverts $S/Q_2$. Now, $[t_E^*, Q_1\cap Q_2, Q_1]=[\Phi(Q_2), Q_1]=\{1\}$ and $[Q_1, (Q_1\cap Q_2), t_E^*]=[\Phi(Q_1), t_E^*]=\{1\}$, since $\Phi(Q_1)\cap \Phi(Q_2)=\{1\}$. Therefore, by the three subgroup lemma, $[t_E^*, Q_1, Q_1\cap Q_2]=\{1\}$ and $t_E^*$ centralizes $Q_1/Q_1\cap Q_2$, a contradiction since $t_E^*$ inverts $S/Q_2\cong Q_1/(Q_1\cap Q_2)$. 

Suppose that $t_E$ does not lift to a morphism of $S$. In particular, we may assume that $Q_2$ is essential. Note that $S$ acts non-trivially on $Z(Q_2)/\Phi(Q_2)$ and centralizes $Z(S)/\Phi(Q_2)$. By \cref{SEFF}, setting $L_2:=O^{3'}(\Out_{\fs}(Q_2))$, we have that $V:=Z(Q_2)/\Phi(Q_2)$ is a natural $\SL_2(q)$-module for $L_2/C_{L_2}(V)\cong\SL_2(q)$ and $C_{L_2}(V)$ is a $3'$-group. Then, independently of a $\mathcal{K}$-group hypothesis, provided $q>3$, \cref{MaxEssen} implies that $L_2$ is a central extension of $\SL_2(q)$ by a group of $3'$-order, and so $L_2\cong\SL_2(q)$. If $q=3$, then \cite[Lemma 7.8]{parkersem} implies that $L_2\cong\SL_2(3)$ and $V$ is a natural $\SL_2(3)$-module. Since $S$ acts non-trivially and quadratically on $Q_2/Z(Q_2)$,  $Q_2/Z(Q_2)$ is also a natural $\SL_2(q)$-module for $L_2$. But then, $L_2$ is transitive on subgroups of $Q_2/Z(Q_2)$ of order $q$ and there is $\alpha\in L_2$ such that $E\alpha=Q_1\cap Q_2$, a contradiction since $E$ is fully normalized. This completes the proof.
\end{proof}

As with the case when $p=2$, we circumvent the need for a $\mathcal{K}$-group hypothesis. As in the above, we only make use of \cref{SEFF} to identify the automizer of an essential subgroup, and this is enough to show that for $E$ an essential subgroup under consideration, $O^{3'}(\Out_{\fs}(E))\cong\SL_2(3^r)$ for some $r$. Moreover, as intimated when $p=2$, under such circumstances the proof of \cref{G2Cor} boils down to recognizing a weak BN-pair of rank $2$ whose completion is completely determined using \cite{Greenbook} which does not rely on any inductive hypothesis. In our application, we identify a specified subsystem of $\fs$ within the fusion category of $\mathrm{G}_2(q)$ using this methodology, and then identify $\fs$ using the relationship between $\Aut(S)$ and $\Aut(\mathrm{G}_2(q))$ demonstrated in \cref{AutS}.

\begin{theorem}
Let $\fs$ be a saturated fusion system over a Sylow $3$-subgroup of $\mathrm{G}_2(3^n)$. Then one of the following occurs:
\begin{enumerate}
\item $\fs=\fs_S(S: \Out_{\fs}(S))$;
\item $\fs=\fs_S(Q_i: \Out_{\fs}(Q_i))$ where $O^{3'}(\Out_{\fs}(Q_i))\cong \SL_2(3^n)$; or
\item $\fs=\fs_S(G)$ where $F^*(G)=O^{3'}(G)\cong\mathrm{G}_2(3^n)$.
\end{enumerate}
\end{theorem}
\begin{proof}
By \cref{Q_1} and \cref{Q1/Q22}, $\mathcal{E}(\fs)\subseteq \{Q_1, Q_2, Q_1\cap Q_2\}$. Suppose that $Q_1\cap Q_2$ is essential. Since $S/Q_1\cap Q_2$ is elementary abelian and of order $q^2$ and $Z(S)$ is of index $q^2$ in $Q_1\cap Q_2$ and centralized by $S$, it follows by \cref{SEFF} that $Q_1\cap Q_2$ is a natural $\SL_2(q^2)$-module for $L_{12}:=O^{3'}(\Out_{\fs}(Q_1\cap Q_2))\cong \SL_2(q^2)$. But then $|N_{L_{12}}(\Out_S(Q_1\cap Q_2))|=q^2-1$ and since $Q_1\cap Q_2$ is receptive, each morphism $\phi\in N_{L_{12}}(\Out_S(Q_1\cap Q_2))$ lifts to some morphism in $\Aut_{\fs}(S)$. Since $N_{\Aut_{\fs}(S)}(Q_1)$ has index at most $2$ in $\Aut_{\fs}(S)$, it follows that upon restriction there is a group of index at least $2$ in $N_{L_{12}}(\Out_{S}(Q_1\cap Q_2))$ normalizing $\Out_{Q_1}(Q_1\cap Q_2)$, a contradiction unless $q=3$. If $q=3$, then $Q_1\cap Q_2$ is not essential in $\fs$ by \cite[Lemma 7.4]{parkersem}. 

We have reduced to the case where the set of essentials is contained in $\{Q_1, Q_2\}$. If neither $Q_1$ nor $Q_2$ is essential then outcome (i) holds. If $Q_i$ is essential then following an argument in \cref{Q1/Q22}, we deduce that $O^{3'}(\Out_{\fs}(Q_i))\cong\SL_2(q)$ and both $Q_i/Z(Q_i)$ and $Z(Q_i)/\Phi(Q_i)$ are natural $\SL_2(q)$-modules. In particular, if only one of $Q_1, Q_2$ is essential then by \cref{swapping}, $\Aut_{\fs}(S)=N_{\Aut_{\fs}(S)}(Q_i)$ and outcome (ii) holds. 

Assume that both $Q_1$ and $Q_2$ are essential and suppose $Q:=O_3(\fs)\ne \{1\}$. By \cref{normalinF}, $Q\le Q_1\cap Q_2$. Then $Q\cap Z(S)\ne\{1\}$ and the irreducibility of $Z(Q_i)/\Phi(Q_i)$ under the action of $O^{3'}(\Out_{\fs}(Q_i))$ implies that $Z(Q_1)Z(Q_1)\le Q_1\cap Q_2\le Q\le Q_1\cap Q_2$ and $Q=Q_1\cap Q_2$. Then, the irreducibility of $O^{3'}(\Out_{\fs}(Q_i))$ on $Q_i/Z(Q_i)$ gives a contradiction. Therefore, $O_3(\fs)=\{1\}$.

Set $\fs_0=\langle N_{\fs}(Q_1), N_{\fs}(Q_2)\rangle$ so that $\Aut_{\fs_0}(S)$ has index at most $2$ in $\Aut_{\fs}(S)$. It follows by \cite[Lemma I.7.6(b)]{ako} that $\fs_0$ is a saturated subsystem of $\fs$ and so $\fs_0$ has index $2$ in $\fs$. In particular, by \cite[Theorem I.7.7(c)]{ako}, $\fs_0$ is a normal subsystem of $\fs$ and $O^{3'}(\fs)\le O^{3'}(\fs_0)$. Now, $\fs_0$ satisfies the hypothesis of \cref{G2Cor} and comparing with the list there, it follows that $O^{3'}(\fs_0)$ is isomorphic to the $3$-fusion system of $\mathrm{G}_2(3^n)$ and since $O^{3'}(\fs_0)$ is simple, we deduce that $O^{3'}(\fs_0)=O^{3'}(\fs)$. By \cref{AutS}, we have that $\Aut(S)=CH$, where $C$ is a $3$-group and $H=N_{\Aut(\mathrm{G}_2(3^n))}(S)$, and so choices of $\Aut_{\fs}(S)$ correspond exactly to $G\le \Aut(\mathrm{G}_2(q))$ such that $F^*(G)=O^{3'}(G)\cong \mathrm{G}_2(q)$, as required.
\end{proof}

\section[Fusion Systems on a Sylow \texorpdfstring{$p$}{p}-subgroup of \texorpdfstring{$\mathrm{G}_2(p^n)$}{G2(pn)} for \texorpdfstring{$p\geq 5$}{p5}]{Fusion Systems on a Sylow $p$-subgroup of $\mathrm{G}_2(p^n)$ for $p\geq 5$}

Suppose now that $p\geq 5$, $q=p^n$ and $S$ is isomorphic to a Sylow $p$-subgroup of $\mathrm{G}_2(q)$. Again, we set $\mathbb{K}$ to be a finite field of order $q$ and recall the Chevalley commutator formulas from \cref{G2Sylow}:

\begin{align*}
[x_\alpha(t), x_\beta(u)]&=x_{\alpha+\beta}(-tu)x_{2\alpha+\beta}(-t^2u)x_{3\alpha+\beta}(t^3u)x_{3\alpha+2\beta}(-2t^3u^2)\\
[x_\alpha(t), x_{\alpha+\beta}(u)]&=x_{2\alpha+\beta}(-2tu)x_{3\alpha+\beta}(3t^2u)x_{3\alpha+2\beta}(3tu^2)\\
[x_\alpha(t), x_{2\alpha+\beta}(u)]&=x_{3\alpha+\beta}(3tu)\\
[x_\beta(t), x_{3\alpha+\beta}(u)]&=x_{3\alpha+2\beta}(tu)\\
[x_{\alpha+\beta}(t), x_{2\alpha+\beta}(u)]&=x_{3\alpha+2\beta}(3tu).
\end{align*}

It then follows that
\[Z_4(S)=S'=\langle X_{\alpha+\beta}, X_{2\alpha+\beta}, X_{3\alpha+\beta}, X_{3\alpha+2\beta}\rangle,\] \[Z_3(S)=S''=\langle X_{2\alpha+\beta}, X_{3\alpha+\beta}, X_{3\alpha+2\beta}\rangle,\] \[Z_2(S)=S'''=\langle X_{3\alpha+\beta}, X_{3\alpha+2\beta}\rangle, and\]\[Z(S)=S^{''''}=S^{(2)}=\langle X_{3\alpha+2\beta}\rangle\] are characteristic subgroups of $S$ of orders $q^4$, $q^3$, $q^2$ and $q$ respectively. In particular, the lower and upper central series for $S$ coincide.

We define
\[Q_1:=C_S(Z_3(S)/Z_1(S))=\langle X_\beta, X_{\alpha+\beta}, X_{2\alpha+\beta}, X_{3\alpha+\beta}, X_{3\alpha+2\beta}\rangle\]
\[Q_2:=C_S(Z_2(S))=\langle X_\alpha, X_{\alpha+\beta}, X_{2\alpha+\beta}, X_{3\alpha+\beta}, X_{3\alpha+2\beta}\rangle\]
both of order $q^5$ and characteristic in $S$. Observe that we may identify $Q_1$ and $Q_2$ with the unipotent radical subgroups of two maximal parabolic subgroups in $\mathrm{G}_2(q)$. Additionally, $\Phi(Q_1)=Z(Q_1)=Z(S)$ and $\Phi(Q_2)=Z_3(S)$.

We first record some useful structural properties of $S$, $Q_1$ and $Q_2$. There is much more to be said here but we only present the results required to prove the \hyperlink{MainTheorem}{Main Theorem}.

\begin{lemma}\label{Q15Iden}
$Q_1$ is isomorphic to $X_1*X_2$ where $Z(S)=Z(X_1)=Z(X_2)$ and $X_i\cong T\in\syl_p(\SL_3(q))$ for $i\in\{1,2\}$.
\end{lemma}
\begin{proof}
Let $X_1=X_{\beta}X_{3\alpha+\beta}X_{3\alpha+2\beta}\le Q_1$. Since the groups $X_\beta$ and $X_{3\alpha+\beta}$ commute modulo $X_{3\alpha+2\beta}$, it follows that every element of $X_1$ may be written in the form $x_{3\alpha+\beta}(t_1)x_\beta(t_2)x_{3\alpha+2\beta}(t_3)$ for $t_i\in\mathbb{K}$. Then, using the commutator formulas, we calculate that the map $\theta_1:X_1\to\SL_3(q)$ such that 
\[
(x_{3\alpha+\beta}(t_1)x_\beta(t_2)x_{3\alpha+2\beta}(t_3))\theta_1=
\begin{pmatrix} 
1 & 0 & 0\\ 
t_1 & 1 & 0 \\ 
t_3 & t_2 &1
\end{pmatrix}
\]
is an injective homomorphism, from which it follows that $X_1$ is isomorphic to a Sylow $3$-subgroup of $\SL_3(q)$. Similarly, letting $X_2=X_{2\alpha+\beta}X_{\alpha+\beta}X_{3\alpha+2\beta}\le Q_1$. Then every element of $X_2$ may be written as $x_{2\alpha+\beta}(t_1)x_{\alpha+\beta}(t_2)x_{3\alpha+2\beta}(t_3)$ for $t_i\in\mathbb{K}$. Then, using the commutator formulas, we calculate that the map $\theta_2:X_2\to\SL_3(q)$ such that 
\[
(x_{2\alpha+\beta}(t_1)x_{\alpha+\beta}(t_2)x_{3\alpha+2\beta}(t_3))\theta_2=
\begin{pmatrix} 
1 & 0 & 0\\ 
t_1 & 1 & 0 \\ 
t_3 & 3t_2 &1
\end{pmatrix}
\]
is an injective homomorphism, from which it follows that $X_2$ is isomorphic to a Sylow $3$-subgroup of $\SL_3(q)$. Thus, $Q_1$ is a central product (over $Z(S)=X_{3\alpha+2\beta}$) of two groups isomorphic to a Sylow $p$-subgroup of $\SL_3(q)$.
\end{proof}

In the literature, $Q_1$ is referred to as an \emph{ultraspecial} group. The properties of such groups are well known.  See, for example, \cite{beisiegel}.

\begin{lemma}\label{5conj}
Let $x\in Z_3(S)\setminus Z_2(S)$. Then $x$ is $S$-conjugate to $x_{2\alpha+\beta}(u)$ for some $u\in\mathbb{K}^\times$.
\end{lemma}
\begin{proof}
Let $x\in Z_3(S)\setminus Z_2(S)$ so that $x=x_{2\alpha+\beta}(t_1)x_{3\alpha+\beta}(t_2)x_{3\alpha+2\beta}(t_3)$ for some $t_1,t_2,t_3\in\mathbb{K}$ with $t_1\ne 0$. Then the element $x_{\beta}(t_3t_2^{-1})x_{\alpha}(3^{-1}t_2t_1^{-1})$ conjugates $x$ to $x_{2\alpha+\beta}(t_1)$ if $t_2\ne 0$ and the element $x_{\alpha+\beta}(3^{-1}t_3 t_1^{-1})$ conjugates $x$ to $x_{2\alpha+\beta}(t_1)$ if $t_2=0$.
\end{proof}

As in the cases where $p=2$ or $3$, the main tool we use to determine whether a subgroup of $S$ is essential is \cref{Chain} and so for a large number of arguments in this section, we need only assume that any essential candidate is $S$-radical and $S$-centric.

\begin{lemma}\label{Q1Q2}
Suppose that $E$ is an $S$-centric, $S$-radical subgroup of $S$ with $Q_1\le E$ or $Q_2\le E$. Then $E\in\{Q_1, Q_2, S\}$.
\end{lemma}
\begin{proof}
Suppose that $Q_1<E$. Then there is $e=x_{\alpha}(t_1)\in E$ with $t_1\ne 0$, applying the commutator formulas, it follows that $Z(E)=Z(S)$, $Z_2(E)=Z_2(S)$, $Z_3(E)=Z_3(S)$ and $E'=S'$. But then $Q_2$ centralizes the chain $\{1\}\normaleq Z_2(E) \normaleq Z_3(E)\normaleq E'\normaleq E$, and since $E$ is $S$-radical, $E=S$. In a similar manner, if $Q_2<E$ then there is $e=x_{\beta}(t_1)\in E$ with $t_1\ne 0$. Again, from the commutator formulas, $Z(E)=Z(S)$ and $E'=S'$. Now, $Q_1$ centralizes the chain $\{1\}\normaleq Z(E)\normaleq E'\normaleq E$ and since $E$ is $S$-radical, $E=S$.
\end{proof}

\begin{lemma}\label{Z3}
Suppose that $E\le S$ is an $S$-centric, $S$-radical subgroup of $S$ with $Z_3(S)=S''\le E$. Then $E=Z_3(S)$ or $Z(E)\le Z_2(S)$. Moreover, if $E$ is essential in some saturated fusion system $\fs$ supported on $S$, then $E\ne Z_3(S)$. 
\begin{proof}
Since $Z_3(S)\le E$ is self-centralizing, we have that $Z(E)\le Z_3(S)$. By \cref{5conj}, if $Z(E)\not\le Z_2(S)$ then there is $e\in Z(E)\setminus Z_2(S)$ with $e$ conjugate in $S$ to some $x_{2\alpha+\beta}(u)$ for $u\in\mathbb{K}^\times$. Thus, $Z_3(S)\le E\le C_S(e)=Z_3(S)(X_\beta)^s$ for some $s\in S$. Suppose that $E>Z_3(S)$. Since $E$ is self centralizing $Z(C_S(e))=Z(S)(X_{2\alpha+\beta})^s\le Z(E)$ and so $Z(E)=Z(C_S(e))$. Therefore, $C_S(e)$ centralizes the series $\{1\}\normaleq Z(E)\normaleq E$ so that $E=C_S(e)\le Q_1$. But now, $Q_1$ centralizes the series $\{1\}\normaleq E'=Z(S)=Q_1'\normaleq E$, a contradiction.

Suppose that $E=Z_3(S)$ is an essential subgroup of some saturated fusion system $\fs$ on $S$. Then $Q_2/E$ is elementary abelian of order $q^2$ and centralizes $Z_2(S)$ which has index $q$ in $E$. Hence, $E$ is an FF-module for $\Out_{\fs}(E)$ so that \cref{SEFF} provides a contradiction.
\end{proof}
\end{lemma}

\begin{lemma}\label{Q1Q2-5}
Suppose that $E$ is an $S$-centric, $S$-radical subgroup of $S$ with $Z_3(S)=S''\le E$ and $Z(E)=Z(S)$. Then $E\in\{Q_1, S\}$.
\end{lemma}
\begin{proof}
Since $Z(E)=Z(S)$, we infer that $E\not\le Q_2$. Moreover, if $E\le Q_1$, then $[E, Q_1]\le Q_1'=Z(S)=Z(E)$ and $Q_1$ centralizes the chain $\{1\}\normaleq Z(E)\normaleq E$. Since $E$ is $S$-radical, it follows that $E=Q_1$ in this case. Hence, we may assume throughout that $E\not\le Q_1,Q_2$ and so there is $e:=x_{\alpha}(t_1)x_{\beta}(t_2)x_{\alpha+\beta}(t_3)\in E$ with $t_1\ne 0\ne t_2$. Then, $[e, Z_2(S)]=Z(S)\le E'$ and $[e, X_{2\alpha+\beta}]Z(S)=Z_2(S)\le E'$. Therefore, $C_E(E')\le E\cap Q_2$.

Suppose first that $[Z_3(S), E']=\{1\}$. Since $Z_3(S)$ is self-centralizing, we have that $Z_2(S)\le E'\le Z_3(S)$. If $E'\ne Z_2(S)$, then $Z_3(S)=C_E(E')$ is a characteristic subgroup of $E$. Then $E\cap Q_1=C_E(Z_3(S)/Z(S))=C_E(Z_3(S)/Z(E))$ is also characteristic in $E$. Then, since $S'$ normalizes $E$, $S'$ centralizes the chain $\{1\}\normaleq Z(E)\normaleq E\cap Q_1\normaleq E$ and since $E$ is radical, $S'\le E$ by \cref{Chain}. But then $E\normaleq S$ and $Q_1$ centralizes the chain $\{1\}\normaleq Z(E)\normaleq E\cap Q_1\normaleq E$ and so $Q_1\le E$. Then by \cref{Q1Q2}, $E=Q_1$ or $E=S$ and since $Z_2(S)\le E'$ and $[E', Z_3(S)]=\{1\}$, we have a contradiction in either case. Therefore, $E'=Z_2(S)$ and $E\cap Q_2=C_E(E')$ is characteristic in $E$.

If $E\cap S'>Z_3(S)$, as $E'=Z_2(S)$, it follows from the commutator formulas that $E\cap Q_2=E\cap S'$. But then $S'$ centralizes the chain $\{1\}\normaleq Z(E)\normaleq E\cap S'\normaleq E$ and since $E$ is $S$-radical, $S'\le E$, $E\normaleq S$ and $S'$ is characteristic in $E$. Now,
$Q_1$ centralizes the chain $\{1\}\normaleq Z(E)\normaleq S'\normaleq E$ so that $Q_1\le E$ and, by \cref{Q1Q2}, $E=S$ or $E=Q_1$. Since $E'=Z_2(S)$, we have a contradiction in either case. Thus, $E\cap S'=Z_3(S)$. If $E\cap Q_2=Z_3(S)$, then $S'$ centralizes the chain $\{1\}\normaleq Z(E)\normaleq Z_3(S)\normaleq E$ and since $E$ is $S$-radical, $S'\le E$. Since $E\cap S'=Z_3(S)$, this is an obvious contradiction. Thus, $Z_3(S)=E\cap S'<E\cap Q_2$. Since $E\not\le Q_2$, there is $e:=x_{\alpha}(t_1)x_{\beta}(t_2)x_{\alpha+\beta}(t_3)\in E$ with $t_2\ne 0$ and $\wt e:=x_{\alpha}(\wt{t_1})x_{\alpha+\beta}(\wt{t_2})\in E\cap Q_2$ with $\wt{t_1}\ne 0$. But then, $[e, \wt e]\not\le Z_2(S)=E'$, a contradiction.

Suppose now that $[Z_3(S), E']\ne \{1\}$. Since $Z_2(S)\le E'$, it follows that there is $x:=x_{\alpha+\beta}(t_1)x_{2\alpha+\beta}(t_2)\in E'$ with $t_1\ne 0$. In particular, $S'\cap E\le C_E(E'/Z(E))\le Q_1\cap E$ and so $S'$ centralizes the chain $\{1\}\normaleq Z(E)\normaleq C_E(E'/Z(E))\normaleq E$ and since $E$ is $S$-radical, $S'\le E$. Therefore, $S'\le C_E(E'/Z(E))$, $E\normaleq S$ and $Q_1$ centralizes the chain $\{1\}\normaleq Z(E)\normaleq C_E(E'/Z(E))\normaleq E$. Since $E$ is $S$-radical, $Q_1\le E$ and since $[Z_3(S), E']\ne\{1\}$, it follows from \cref{Q1Q2} that $E=S$.
\end{proof}

\begin{lemma}\label{AZ}
Suppose that $E$ is an $S$-centric, $S$-radical subgroup of $S$ with $Z_3(S)=S''< E$ and $Z(E)\ne Z(S)$. Then $E=Q_2$; or $E\le Q_2$ has order $q^4$, $\Phi(E)< Z_2(S)=Z(E)$, $|\Phi(E)|=q$ and $N_S(E)=Q_2$. Moreover, if $E$ is essential in some saturated fusion system $\fs$ supported on $S$ then $E=Q_2$.
\end{lemma}
\begin{proof}
By \cref{Z3}, $Z(S)<Z(E)\le Z_2(S)$. Then $E\le Q_2$ and $Z(E)=Z_2(S)$ is characteristic in $E$. If $S'=E$ then $Q_1$ centralizes the chain $\{1\}\normaleq E'\normaleq E$, a contradiction since $E$ is assumed to be $S$-radical; and if $S'<E$, by the commutator formulas, it follows that $Z_2(E)=Z_3(S)=Q_2'$ is characteristic in $E$ and so $Q_2$ centralizes the chain $\{1\}\normaleq Z(E)\normaleq Z_2(E) \normaleq E$ and as $E$ is $S$-radical, $E=Q_2$ in this case. Hence, $S'\not\le E$. Moreover, if $E\le S'$ then $S'$ centralizes the series $\{1\}\normaleq Z(E)\normaleq E$ so $E\not\le S'$. Suppose there exists $x\in (S'\cap E) \setminus Z_3(S)$ and let $e\in E\setminus S'$. Since $Z_3(S)\le S'\cap E$, we may take $x=x_{\alpha+\beta}(t_1)$. Then $Z(S)=[x, Z_3(S)]\le E'$ and $Z_2(S)=Z(S)[e,Z_3(S)]\le E'$. Thus, $Z_2(S)<Z_2(S)[e,x]\le E'\le Z_3(S)$, $C_E(E')=Z_3(S)$ is characteristic in $E$ and $S'$ centralizes the chain $\{1\}\normaleq Z(E)\normaleq C_E(E')\normaleq E$, a contradiction since $E$ is $S$-radical. Hence, $S'\cap E=Z_3(S)$ and since $S'E\le Q_2$, $|E|\leq q^4$. Moreover, comparing with commutator formulas, it follows that $N_S(E)=Q_2$.

Now, analyzing $Q_2$ within $\mathrm{G}_2(q)$, we see that $Q_2/Z_3(S)$ is a natural $\SL_2(q)$-module for $O^{p'}(\Out_{\mathrm{G}_2(q)}(Q_2))\cong\SL_2(q)$. In particular, $E$ is contained in some subgroup $X$ of order $q^4$ such that $X$ is conjugate in $O^{p'}(\Out_{\mathrm{G}_2(q)}(Q_2))$ to $S'$. Since $S^{(2)}=Z(S)$, and $Z_2(S)$ is also a natural $\SL_2(q)$-module for $O^{p'}(\Out_{\mathrm{G}_2(q)}(Q_2))\cong\SL_2(q)$, it follows that $\Phi(X)$ is a group of order $q$ contained in $Z_2(S)=Z(E)$. In particular, if $E<X$, then $X$ centralizes the chain $\{1\}\normaleq Z(E)\normaleq E$, a contradiction since $E$ is $S$-radical. Therefore, $E=X$ is of order $q^4$ and satisfies the required properties. 

Assume now that $E$ is essential in a saturated fusion system $\fs$ on $S$. By the results in \cite[Lemma 4.4]{parkersem}, we may assume that $q>p$ else the result holds. Note that $Q_2$ centralizes $Z_2(S)$ and since $Q_2=N_S(E)$, $O^{p'}(\Out_{\fs}(E))$ centralizes $Z_2(S)=Z(E)$. Moreover, since $\Phi(E)\le Z_2(S)$, $|Q_2/E|= q$,$|E/Z_3(S)|=q$ and $[Q_2, Z_3(S)]=Z_2(S)$,  it follows by a similar argument to \cref{SEFF} that $E/Z(E)$ is a natural $\SL_2(q)$-module for $O^{p'}(\Out_{\fs}(E))\cong \SL_2(q)$. 

Suppose first that $Q_2$ is essential in $\fs$. By \cref{Q1Q2}, $Q_2$ is maximally essential. Since $\Phi(Q_2)=Z_3(S)$ and $[S, S']\le Z_3(S)$, by \cref{SEFF} we deduce that $Q_2/\Phi(Q_2)$ is a natural $\SL_2(q)$-module for $O^{p'}(\Out_{\fs}(Q_2))\cong\SL_2(q)$. But then, $O^{p'}(\Out_{\fs}(Q_2))$ acts transitively on subgroups of $Q_2$ of order $q^4$ containing $\Phi(Q_2)=Z_3(S)$ so that $E$ is conjugate in $\fs$ to $S'$. Since $E$ was assumed to be fully $\fs$-normalized, this is a contradiction.

Hence, we may assume that $Q_2$ is not essential. Note that as any essential containing $E$ contains $Z_3(S)$, we may as well assume that $E$ is not properly contained in any essential subgroup and so $E$ is maximally essential. Let $t_E$ be a non-trivial element in $Z(O^{p'}(\Out_{\fs}(E)))$. Using that $t_E$ normalizes $\Out_S(E)$, $E$ is receptive and applying the Alperin--Goldschmidt theorem, $t_E$ lifts to some morphism in $\Aut_{\fs}(S)$ and so normalizes $Z_3(S)$ and $S'$. Moreover, since $E/Z(E)$ is natural $\SL_2(q)$-module, $t_E$ inverts $Z_3(S)/Z(E)$, centralizes $Z(E)$ and centralizes $Q_2E/E\cong S'/Z_3(S)$. But now, $[t_E, S', Z_3(S)]=\{1\}$ since $Z_3(S)$ is abelian, and $[S', Z_3(S), t_E]=\{1\}$. By the three subgroup lemma, $[t_E, Z_3(S), S']=\{1\}$ and so $[t_E, Z_3(S)]\le Z(S')=Z_2(S)=Z(E)$, a contradiction. Therefore, $E$ is not essential in $\fs$.
\end{proof}

\begin{lemma}\label{EAEssen}
Suppose that $E$ is an $S$-centric, $S$-radical subgroup of $S$ with $Z_3(S)\not\le E$ but $Z_2(S)\le E$. Then $E\cap Z_3(S)=Z_2(S)$.
\end{lemma}
\begin{proof}
Since $Z_2(S)\le E$, we deduce that $Z(E)\le Q_2$. Suppose that $E\cap Z_3(S)>Z_2(S)$. Since $Z(E)$ centralizes $E\cap Z_3(S)$ and $Z_3(S)$ is self-centralizing in $S$, it follows that $Z(E)\le Z_3(S)$. If $Z(E)\cap Z_2(S)>Z(S)$, then $E\le Q_2$ and $Z_2(S)\le Z(E)\le Z_3(S)$. Moreover, if $Z_2(S)<Z(E)$ then, again using that $Z_3(S)$ is self-centralizing, it follows that $E\le Z_3(S)$ and since $E$ is $S$-centric, $E=Z_3(S)$, a contradiction. Hence, if $Z(E)\cap Z_2(S)>Z(S)$ then $Z(E)=Z_2(S)$. But now, $Z_3(S)$ centralizes the chain $\{1\}\normaleq Z(E)\normaleq E$, a contradiction since $E$ is $S$-radical and $Z_3(S)\not\le E$. Therefore, if $E\cap Z_3(S)>Z_2(S)$, then $Z(E)\cap Z_2(S)=Z(S)$.

Suppose that $Z(E)\cap Z_3(S)>Z(S)$ and let $e\in (Z_3(S)\cap Z(E))\setminus Z(S)$. By \cref{5conj}, $e$ is conjugate in $S$ to some element $x_{2\alpha+\beta}(t)$ with $t\ne 0$. Moreover, it follows from the commutator formulas that the centralizer of such an element is contained in $Q_1$ and intersects $S'$ only in $Z_3(S)$. Since $Q_1$, $S'$ and $Z_3(S)$ are normal in $S$, the centralizer of $e$ is contained in $Q_1$ and intersects $S'$ only in $Z_3(S)$. But $E$ centralizes $e\le Z(E)$ and so if $E\le S'$, then $E\le Z_3(S)$ and since $E$ is $S$-centric, we have a contradiction. Therefore, $E\le Q_1$ and there is $x\in E\setminus S'$. Since $Z_2(S)\le E$, $Z(S)=[x, Z_2(S)]\le E' \le Q_1'=Z(S)$ and so $Z(S)=E'$. But then $Q_1$ normalizers the chain $\{1\}\normaleq E'\normaleq E$, and since $E$ is $S$-radical, we conclude that $Z_3(S)\le Q_1\le E$, a contradiction.

Hence, we have shown that if $E\cap Z_3(S)>Z_2(S)$, then $Z(E)=Z(S)$. In particular, $E\not\le Q_2$ since $Z_2(S)\not\le Z(E)$ and $E\not\le Q_1$, for otherwise $Q_1$ centralizes the chain $\{1\}\normaleq Z(E)\normaleq E$, a contradiction for then $Z_3(S)\le Q_1\le E$ since $E$ is $S$-radical. Now, $Z_2(S)\le Z_2(E)$ and since $E\cap Z_3(S)>Z_2(S)$, it follows from the commutator formulas that $Z_2(E)\le E \cap Q_1$. But then $[Z_3(S), Z_2(E)]\le Z(S)=Z(E)$, $[Z_3(S), E]\le Z_2(S)\le Z_2(E)$ and $Z_3(S)$ centralizes the chain $\{1\}\normaleq Z(E)\normaleq Z_2(E)\normaleq E$, a contradiction since $E$ is $S$-radical.
\end{proof}

\begin{lemma}
Suppose that $E$ is an $S$-centric, $S$-radical subgroup of $S$ with $Z_3(S)\not\le E$ but $Z_2(S)\le E$. Then either $E\le S'$ is elementary abelian of order $q^3$, $N_S(E)=Q_1$ and $E$ is not an essential subgroup of any saturated fusion system $\fs$ on $S$; or $E\cap S'=Z_2(S)$. 
\end{lemma}
\begin{proof}
By \cref{EAEssen}, we have that $E\cap Z_3(S)=Z_2(S)$. We assume throughout that $E\cap S'>Z_2(S)$. It then follows from the commutator formulas that $Z(E)\le S'$. We suppose first that $Z(E)=Z(S)$. In particular, $E\not\le Q_2$ and since $Z_2(S)\le E$, for $x\in E\setminus Q_2$, $Z(S)=[x, Z_2(S)]\le E'$. Now, for $e\in (E\cap S')\setminus Z_3(S)$, $[e, Z_2(E)]=Z(E)$ and it follows from the commutator formulas that $Z_2(S)\le Z_2(E)\le Q_1$. In particular, $Z_3(S)$ centralizes the chain $\{1\}\normaleq Z(E)\normaleq Z_2(E)\normaleq E$, a contradiction since $E$ is $S$-radical and $Z_3(S)\not\le E$.

Suppose now that $Z(E)>Z(S)$ but $Z(E)\cap Z_2(S)=Z(S)$. Then there is $e\in (Z(E)\cap S')\setminus Z(S)$ and it follows from the commutator formulas that the centralizer of such an element is contained in $Q_1$. Therefore, $E\le Q_1$ and $E'\le Q_1'=Z(S)$. Moreover, if there is $x\in E\setminus S'$, then $Z(S)=[x, Z_2(S)]\le E'=Z(S)$ and so, $Q_1$ centralizes the chain $\{1\}\normaleq E' \normaleq E$, a contradiction since $Q_1\not\le E$ and $E$ is $S$-radical. Therefore, $E\le S'$, which yields another contradiction for then $Z_2(S)\le Z(E)$. 

Finally, suppose that $Z(E)\cap Z_2(S)>Z(S)$ so that $E\le Q_2$. Then $Z_2(S)\le Z(E)$ and since $Z_3(S)\not\le E$ and $E$ is $S$-radical, we conclude that $Z_2(S)<Z(E)$ for otherwise, $Z_3(S)$ centralizes the chain $\{1\}\normaleq Z(E)\normaleq E$. But then, there is $e\in (Z(E)\cap S')\setminus Z_3(S)$ and by the commutator formulas, $E\le S'$. Since $E\cap Z_3(S)=Z_2(S)$, we deduce that $|E|\leq q^3$. Set $e:=x_{\alpha+\beta}(t_1)x_{2\alpha+\beta}(t_2)x$, where $x\in Z_2(S)$ and $t_1\in\mathbb{K}^\times$. Then for $y:=x_{\alpha}(-t_22^{-1}t^{-1})$, $e^y=x_{\alpha+\beta}(t_1)x'$ for some $x'\in Z_2(S)$. Then $C_S(e^yZ_2(S))=X_{\alpha+\beta}Z_2(S)$ and it follows that $E\le C_S(e)$ is conjugate to a subgroup of $X_{\alpha+\beta}Z_2(S)$. Moreover, since $E$ is $S$-centric and $X_{\alpha+\beta}Z_2(S)$ is elementary abelian, $E$ is conjugate to $X_{\alpha+\beta}Z_2(S)$ and a calculation using the commutator formulas gives that $N_S(E)=Q_1$.

In this scenario, assume that $E$ is essential in some saturated fusion system $\fs$ on $S$. Since $Z_3(S)E/E$ is elementary abelian of order $q$ and $Z_3(S)$ centralizes $Z_2(S)$ which has index $q$ in $E$, by \cref{SEFF} we deduce that $E/C_E(O^{p'}(\Out_{\fs}(E)))$ is a natural $\SL_2(q)$-module for $O^{p'}(\Out_{\fs}(E))\cong \SL_2(q)$ and $\Out_{Z_3(S)}(E)\in\syl_p(\Out_{\fs}(E))$. But $Q_1\le N_S(E)$ and we have a contradiction.
\end{proof}

\begin{lemma}\label{EmeetZ2}
Suppose that $E$ is an $S$-centric, $S$-radical subgroup of $S$ with $E\cap S'=Z_2(S)$. Then either
\begin{enumerate}
\item $E\le Q_2$ is elementary abelian of order $q^3$, $E\not\le S'$ and $N_S(E)=EZ_3(S)$ has order $q^4$; or
\item $E\cong q^{1+2}$, $Z_2(S)=E\cap Q_1=E\cap Q_2$, $Z(S)=Z(E)=\Phi(E)$ and $N_S(E)=EZ_3(S)$ has order $q^4$.
\end{enumerate}
Moreover, in both cases $E$ is not essential in any saturated fusion system $\fs$ on $S$.
\end{lemma}
\begin{proof}
Suppose first that $E\le Q_2$. Then $Z_2(S)\le Z(E)$ and since $E\cap S'=Z_"(S)$, $|E|\leq q^3$. If $Z(E)=Z_2(S)$, then $Z_3(S)$ centralizes the chain $\{1\}\normaleq Z(E)\normaleq E$, a contradiction since $E$ is $S$-radical. Therefore, there is $e\in Z(E)\setminus S'$ and we may write $e=x_{\alpha}(t_1)x_{\alpha+\beta}(t_2)x_{3\alpha+\beta}(t_3)x$ for some $x\in Z_2(S)$, $t_1\in \mathbb{K}^\times$ and $t_2, t_3\in\mathbb{K}$. Then for $y:=x_{\beta}(t_1^{-1}t_2)x_{\alpha+\beta}(2^{-1}t_1^{_1}(t_3-t_1t_2))$, we have that $e^y=x_{\alpha}(t_1)x'$ for some $x'\in Z_2(S)$. Then $C_S(e^yZ_2(S))=X_\alpha Z_2(S)$ and by conjugation, $E\le C_S(e)$ is conjugate to a subgroup of $X_\alpha Z_2(S)$. Moreover, since $E$ is $S$-centric and $X_{\alpha}Z_2(S)$ is elementary abelian, we conclude that $E$ is conjugate to $X_{\alpha}Z_2(S)$ and a calculation using the commutator formulas gives that $N_S(E)=EZ_3(S)$, as required.

Suppose now that $E$ is essential in a saturated fusion system $\fs$ on $S$. Then $Z_3(S)E/E$ is elementary abelian of order $q$ and $Z_3(S)$ centralizes $Z_2(S)$ which has index $q$ in $E$. By \cref{SEFF}, $E/C_E(O^{p'}(\Out_{\fs}(E)))$ is a natural $\SL_2(q)$-module for $O^{p'}(\Out_{\fs}(E))\cong \SL_2(q)$ and $\Out_{Z_3(S)}(E)\in\syl_p(\Out_{\fs}(E))$. Since $E\not\le Q_1$, we may assume by \cref{Q1Q2-5} and \cref{AZ} that the only possible essential $E$ is properly contained in is $Q_2$. 

If $Q_2$ is essential then using that $S$ centralizes $S'/Z_3(S)=S'/\Phi(Q_2)$ and $S'/Z_3(S)$ has index $q$ in $Q_2/Z_3(S)$, it follows by \cref{SEFF} that $Q_2/Z_3(S)$ is a natural $\SL_2(q)$-module for $O^{p'}(\Out_{\fs}(Q_2))\cong\SL_2(q)$. But then, $O^{p'}(\Out_{\fs}(Q_2))$ is transitive on subgroup of order $q$ in $Q_2/\Phi(Q_2)$ and so $E\phi \le S'$ for some $\phi\in O^{p'}(\Out_{\fs}(Q_2))$. Therefore, $[E\phi, Q_1]\le Z(S)\le Z_2(S)\le E\phi$ and $Q_1\le N_S(E\phi)$. Since $|N_S(E)|=q^4$, $E$ is not fully normalized, a contradiction. 

Hence, we may assume that $Q_2$ is not essential in $\fs$ and for a non-trivial element $t_E\in Z(O^{p'}(\Out_{\fs}(E)))$, using that $E$ is receptive, $t_E$ lifts to some $t_E^*\in\Aut_{\fs}(S)$. Moreover, by coprime action, $E=[E, t_E^*]\times C_E(t_E^*)$ and either $Z(S)=C_E(t_E^*)$ or $Z(S)\cap C_E(t_E^*)=\{1\}$. Since $Z_2(S)=C_E(Z_3(S))$, it follows in the latter case that $t_E^*$ centralizes $Z_2(S)/Z(S)$ and since $Z_3(S)E/E\cong Z_3(S)/Z_2(S)$, coprime action yields $[Z_3(S), t_E^*]=Z(S)$. Then, $[Z_3(S), S, t_E^*]=Z(S)$, $[t_E^*, Z_3(S), S]=\{1\}$ and the three subgroup lemma yields, $[S, t_E^*, Z_3(S)]\le Z(S)$ and $t_E^*$ centralizes $S/Q_1\cong Q_2/S'=ES'/S'\cong E/Z_2(S)$, a contradiction. Thus, $t_E^*$ centralizes $Z(S)$ and inverts $Z_2(S)/Z(S)$. Moreover, $t_E^*$ centralizes $Z_3(S)/Z_2(S)$ and inverts $E/Z_2(S)=E/E\cap S'\cong Q_2/S'$. Now, since $[S', Z_3(S)]\le Z(S)$ is centralized by $t_E^*$ and $[Z_3(S), t_E^*]\le Z_2(S)$ is centralized by $S'$, it follows from the three subgroup lemma that $[t_E^*, S', Z_3(S)]=\{1\}$ and since $Z_3(S)$ is self-centralizing, $[t_E^*, S']\le Z_3(S)$. Indeed, coprime action implies that $[t_E^*, S']\le Z_2(S)$. But then $[t_E^*, S', Q_2]=\{1\}$, $[S', Q_2, t_E^*]\le Z_2(S)$ and another application of the three subgroup lemma gives $[t_E^*, Q_2, S']\le Z_2(S)$. But $t_E^*$ inverts $Q_2/S'$ and a contradiction follows from the commutator formulas.

Assume now that $E\not\le Q_2$ and since $Z_2(S)\le E$, for $x\in E\setminus Q_2$, we have that $Z(S)=[x, Z_2(S)]\le E'\le E\cap S'=Z_2(S)$. If $Z(S)<E'$, then $C_E(E')=E\cap Q_2$ is characteristic in $E$. Moreover, $Z_2(S)<C_E(E')$ for otherwise $Z_3(S)$ centralizes the chain $\{1\}\normaleq Z_2(S)\normaleq E$, a contradiction since $Z_3(S)\not\le E$ and $E$ is $S$-radical. Furthermore, $Z(E)\cap Q_2\le S'\cap E=Z_2(S)$, otherwise $E\le Q_2$. But then $Z(E)=Z(S)$ and since there is $e\in (E\cap Q_2)\setminus Z_2(S)$, $Z_2(S)\le Z_2(E)\le E\cap Q_1$ and so $Z_2(S)=Z_2(E)\cap (E\cap Q_2)$ is characteristic in $E$ and $Z_3(S)$ centralizes the chain $\{1\}\normaleq Z_2(S)\normaleq E$, a contradiction.

Finally, we suppose that $E\cap S'=Z_2(S)$, $E\not\le Q_2$ and $Z(S)=E'$. If $E\cap Q_2>Z_2(S)$ then, as $E\not\le Q_2$, there is $e\in E\setminus Q_2$, with $[e, E\cap Q_2]\not \le Z(S)=E'$. Hence, $E\cap Q_2=Z_2(S)$ and $|E|\leq q^3$. Notice that if $E\le Q_1$, then $[E, Q_1]\le Q_1'=Z(S)=E'$ and $Q_1$ centralizes the chain $\{1\}\normaleq E'\normaleq E$, a contradiction since $Z_3(S)\not\le E$ and $E $ is $S$-radical. Hence, there is $e\in E\setminus (Q_1\cup Q_2)$ and since $[e, E\cap Q_1]\le E'=Z(S)$, it follows from the commutator formulas that $E\cap Q_1=Z_2(S)$. Note that $EQ_1/Q_1\cong E/Z_2(S)$ is elementary abelian and so, $\Phi(E)\le Z_2(S)$. If $Z(S)<\Phi(E)$, then $Z_2(S)=C_E(\Phi(E))$ is characteristic in $E$, a contradiction for then $Z_3(S)$ centralizes then $\{1\}\normaleq Z_2(S)\normaleq E$. Therefore, $\Phi(E)=Z(E)=Z(S)$, $|E|=q^3$ and the commutator formulas imply that $N_S(E)=Z_3(S)E$, as required.

Suppose that $E$ is essential in some saturated fusion system $\fs$ supported on $S$. Since $E\not\le Q_1, Q_2$, it follows by \cref{Q1Q2-5} and \cref{AZ} that $E$ is maximally essential. Moreover, $Z_3(S)E/E$ is elementary abelian of order $q$ and $Z_3(S)$ centralizes $Z_2(S)$ which has index $q$ in $E$. Then by \cref{SEFF}, $E/Z(E)$ is a natural $\SL_2(q)$-module, $O^{p'}(\Out_{\fs}(E))\cong \SL_2(q)$ and $\Out_{Z_3(S)}(E)\in\syl_p(\Out_{\fs}(E))$.

Let $\lambda\in N_{O^{p'}(\Out_{\fs}(E)}(\Out_S(E))$ be an element of order $q-1$, isomorphic to a generator of a torus in $\SL_2(q)$. We can choose $\lambda$ to act as the scalar $\mu^{-1}$ on $E/Z_2(S)$ and as $\mu$ on $Z_2(S)/Z(S)$, for $\mu\in\mathbb{K}^\times$. Since $E$ is essential, it is receptive, so we may extend $\lambda$ to some $\hat{\lambda}$, and by the Alperin -- Goldschmidt Theorem and since $E$ is maximally essential, we may take $\hat{\lambda}\in\Aut_{\fs}(S)$ so that $\hat{\lambda}$ acts on $S', Q_1$ and $Q_2$. Since $E/Z_2(S)\cong ES'/S'$, it follows that $\hat{\lambda}$ acts as $\mu^{-1}$ on $ES'/S'$. Let $x_{\alpha}(t_1), x_{\beta}(t_2)$ be transversals in $S/S'$ such that $x_{\alpha}(t_1)x_{\beta}(t_2)S'\in ES'/S'$. We have that \[x_{\alpha}(t)\hat{\lambda}=(x_{\alpha}(t)x_{\beta}(u)\hat{\lambda})(x_{\beta}(-u)\hat{\lambda})=(x_{\alpha}(\mu^{-1}t) x_{\beta}(\mu^{-1}u)(x_{\beta}(-u)\hat{\lambda})\] and comparing coefficients, we have that $\hat{\lambda}$ acts as $\mu^{-1}$ on both $Q_1/S'$ and $Q_2/S'$. Then, by the commutator formula \[[x_{\alpha}(t), x_{\alpha+\beta}(u)]=x_{2\alpha+\beta}(-2tu)x_{3\alpha+\beta}(3t^2u)x_{3\alpha+2\beta}(3tu^2)\] and using that $\hat{\lambda}$ acts as $\mu^2$ on $N_S(E)/E\cong Z_3(S)/Z_2(S)$, we deduce that $\hat{\lambda}$ acts as $\mu^3$ on $S'/Z_3(S)$. Using the commutator relation $[x_{\alpha+\beta}(t), x_{2\alpha+\beta}(u)]=x_{3\alpha+2\beta}(3tu)$ we get that $\hat{\lambda}$ acts as $\mu^5$ on $Z(S)$. But since $Z(S)=C_{E}(O^{p'}(\Out_{\fs}(E)))$ and since $\lambda$ was of order $q-1$, it follows that $q=6$, a contradiction.
\end{proof}

Given \cref{Q1Q2-5}, \cref{AZ} and \cref{EmeetZ2}, we finally assume that $Z_2(S)\not\le E$. This is a particularly interesting case as there is some exceptional behaviour when $q=p=7$ related to the $7$-fusion system of the Monster sporadic simple group. Indeed, this exceptional behaviour produces a distinct class of essentials and with it, a large number of exotic fusion systems. This phenomena was already known about by the work in \cite{parkersem}. 

\begin{lemma}\label{5Essential}
Suppose that $E$ is an $S$-centric, $S$-radical subgroup of $S$ with $Z_2(S)\not\le E$. Then either
\begin{enumerate}
\item $E\le Q_1$ is elementary abelian of order $q^3$, $E\not\le S'$ and $N_S(E)=Q_1$; or
\item$p\geq 7$, $E$ is elementary abelian of order $q^2$, $E\cap Q_1=E\cap Q_2=Z(S)$ and $N_S(E)=Z_2(S)E$.
\end{enumerate}
\end{lemma}
\begin{proof}
We may suppose $Z(E)\not\le Q_2$ for otherwise $Z_2(S)$ centralizes the chain $\{1\}\normaleq Z(E)\normaleq E$, a contradiction since $Z_2(S)\not\le E$ and $E$ is $S$-radical. In particular, it follows by the commutator formulas that $E\cap Q_2 \le S'$ and $E\cap Z_2(S)=Z(S)$. 

Suppose that $E\cap Q_1\ne Z(S)$. Then a calculation using the commutator formulas reveals that $Z(E)\le Q_1$. Then, $Z(E)\not\le S'$ for otherwise $Z_2(S)$ centralizes the chain $\{1\}\normaleq Z(E)\normaleq E$, and another calculation yields $E\le Q_1$. Recall from \cref{Q15Iden} that $Q_1\cong q^{1+2}*q^{1+2}$. Then, $m_p(Q_1)=3n$ and for any element of order $x\in Q_1\setminus Z(S)$ of order $p$, we have that $|C_{Q_1}(x)|=q^4$, $|Z(C_S(e))|=q^2$ and $C_S(e)'=Z(S)$. Since $Z(E)\not\le Q_2$, there is $e\in Z(E)$ such that $E\le C_S(e)$ where $C_S(e)$ has order at most $q^4$. Then, as $E$ is $S$-centric, $Z(C_S(e))\le Z(E)$. Now, if $Z(E)=Z(C_S(e))$, then $C_S(e)$ centralizes the chain $\{1\}\normaleq Z(E)\normaleq E$, and since $E$ is $S$-radical, $E=C_S(e)$. But then $Q_1$ centralizes the chain $\{1\}\normaleq  E'\normaleq E$, a contradiction since $Z_2(S)\not\le E$. 

So assume that $Z(C_S(e))<Z(E)$. It follows that there is $e'\in (Z(E)\cap S')\setminus Z(S)$ so that $E\le C_S(e')$ and again $Z(C_S(e'))\le Z(E)$. Thus, $Z(C_S(e'))Z(C_S(e))$ is elementary abelian of order $q^3$ and contained in $Z(E)$. But $m_p(Q_1)=3n$ and so $E=Z(E)=Z(C_S(e'))Z(C_S(e))$ is elementary abelian of order $q^3$. It follows directly from the commutator formulas that $N_S(E)=Q_1$.

Thus, we have shown that $Z(S)=E\cap Q_1=E\cap Q_2$ and $|E|\leq q^2$. If $p\geq 7$, then as $S$ has exponent $p$ and $E$ is centric, we can explicitly construct elementary abelian subgroups of order $q$ complementing $Z(S)$ in $E$ so that $E=\Omega(Z(E))$ is of order $q^2$. If $p=5$, then $S$ has exponent $25$ and it follows that $\mho(E)=E\cap S'=Z(S)$ and $Z_2(S)$ centralizes the chain $\{1\}\normaleq \mho(E)\normaleq E$, a contradiction since $E$ is $S$-radical.
\end{proof}

\begin{lemma}\label{ExoticEssential}
Suppose that $E\le S$ is an essential subgroup of $\fs$ and $Z_2(S)\not\le E$. Then $q=p=7$ and $E=\langle Z(S), x\rangle$ for some $x\in S\setminus (Q_1\cup Q_2)$.
\end{lemma}
\begin{proof}
By \cref{5Essential}, we may assume that $E$ is elementary abelian of order $q^3$ and contained in $Q_1$; or $E$ is elementary abelian of order $q^2$ and intersects $Q_1$ only in $Z(S)$. In the former case, $Z_2(S)E/E$ is elementary abelian of order $q$ and $Z_2(S)$ centralizes $E\cap S'$ which has index $q$ in $E$. Then by \cref{SEFF}, it follows that $E/C_E(O^{p'}(\Out_{\fs}(E)))$ is a natural $\SL_2(q)$-module for $O^{p'}(\Out_{\fs}(E))\cong \SL_2(q)$. But $N_S(E)=Q_1$ and $|Q_1/E|=q^2$, a contradiction. 

Thus, $E$ is elementary abelian of order $q^2$ and $E\cap Q_1=E\cap Q_2=Z(S)$. Since $Z_2(S)$ centralizes $Z(S)$ which has index $q$ in $E$, by \cref{SEFF}, $E$ is a natural $\SL_2(q)$-module for $O^{p'}(\Out_{\fs}(E))\cong\SL_2(q)$ and $\Out_{Z_2(S)}(E)=\Out_S(E)$. By \cref{Q1Q2-5}, \cref{AZ} and \cref{EmeetZ2} and since $E\not\le Q_1,Q_2$, we assume that $E$ is maximally essential.

Let $\lambda\in N_{O^{p'}(\Out_{\fs}(E))}(\Out_S(E))$ be an element of order $q-1$, isomorphic to a generator of a torus in $\SL_2(q)$. Since $E$ is a natural $\SL_2(q)$-module, for some $\mu\in K^\times$ of order $q-1$, we can choose $\lambda$ to acts as $\mu$ on $Z(S)$ and $\mu^{-1}$ on $E/Z(S)$. Since $E$ is receptive, and by the Alperin--Goldschmidt Theorem, $\lambda$ extends to $\hat{\lambda}\in \Aut_{\fs}(S)$. Since $Q_1, Q_2, S'$ are characteristic in $S$, $\lambda$ acts on $Q_1/S'$, $Q_2/S'$ and $ES'/S'\cong E/Z(S)$. Let $x_{\alpha}(t)$ be a transversal of $Q_2/S'$. Then  $x_{\alpha}(t)\hat{\lambda}=(x_{\alpha}(t)x_{\beta}(u)x_{\beta}(-u))\hat{\lambda}$ for all $u\in K^\times$. But, for some $u$, $x_{\alpha}(t)x_{\beta}(u)$ is a transversal of $ES'/S'$ and $x_{\beta}(-u)$ is a transversal of $Q_1/S'$ and $\hat{\lambda}$ acts on $ES'/S'$ as $\mu^{-1}$.

Thus, \[x_{\alpha}(t)\hat{\lambda}=(x_{\alpha}(t)x_{\beta}(u)\hat{\lambda})(x_{\beta}(-u)\hat{\lambda})=(x_{\alpha}(\mu^{-1}t) x_{\beta}(\mu^{-1}u)(x_{\beta}(-u)\hat{\lambda})\] and by comparing coefficients, $\hat{\lambda}$ acts as $\mu^{-1}$ on both $Q_1/S'$ and $Q_2/S'$. Using the commutator formulas on various elements on $S$, one has that $\hat{\lambda}$ acts as $\mu^{-2}$, $\mu^{-3}$, $\mu^{-4}$ and $\mu^{-5}$ on $S'/Z_3(S)$, $Z_3(S)/Z_2(S)$, $Z_2(S)$ and $Z(S)$ respectively. But since $\hat{\lambda}$ acts on $Z(S)$ as $\lambda$ does, $\mu^{-5}=\mu$ and $\mu^6=1$. Since $\mu$ was of order $q-1$, we conclude that $q=p=7$. In this case, $S$ has exponent $7$ and there is $x\in E\setminus (Q_1\cup Q_2)$ of order $7$ such that $E=\langle Z(S), x\rangle$, as required.
\end{proof}

Before determining all possible saturated fusion systems on $S$, we sum up the results concerning $S$-centric, $S$-radical subgroups of $S$.

\begin{proposition}
Suppose that $E$ is an $S$-centric, $S$-radical subgroup of $S$. Then one of the following holds:
\begin{enumerate}
\item $E\in\{Q_1,Q_2,S\}$;
\item $E\le Q_2$ has order $q^4$, $\Phi(E)< Z_2(S)=Z(E)$, $|\Phi(E)|=q$ and $N_S(E)=Q_2$;
\item $E\le S'$ is elementary abelian of order $q^3$ with $E\normaleq S$ if $E=Z_3(S)$; and $N_S(E)=Q_1$ otherwise;
\item $E\le Q_2$ is elementary abelian of order $q^3$, $E\not\le S'$ and $N_S(E)=EZ_3(S)$ has order $q^4$;
\item $E\cong q^{1+2}$, $Z_2(S)=E\cap Q_1=E\cap Q_2$, $Z(S)=Z(E)=\Phi(E)$;
\item $E\le Q_1$ is elementary abelian of order $q^3$, $E\cap Z_2(S)=Z(S)$ and $N_S(E)=Q_1$; or
\item $p\geq 7$, $E$ is elementary abelian of order $q^2$, $Z(S)=E\cap Q_1=E\cap Q_2=Z(S)$ and $N_S(E)=EZ_2(S)$ has order $q^3$.
\end{enumerate}
\end{proposition}

We now analyze the automizers of the potential essential subgroups of a saturated fusion system $\fs$ over $S$. That is, $Q_1, Q_2$ and if $q=p=7$, some conjugacy class of elementary abelian subgroups of order $7^2$. For the latter class of essentials, we refer to \cite{parkersem} to determine the fusion system, where a large number of exotic fusion systems are uncovered. We analyze the automizer of $Q_2$ via \cref{SEFF}, noting that this result is independent of a $\mathcal{K}$-group hypothesis. Analyzing the automizer of $Q_1$ is more complicated and, with the help of some supporting results, we conclude that $O^{p'}(\Out_{\fs}(Q_1))$ is isomorphic to a subgroup of $\Sp_4(q)$. Since the maximal subgroups of $\Sp_4(q)$ are known by \cite{Mitchell}, we compute the candidates for $O^{p'}(\Out_{\fs}(Q_1))$ independent of any $\mathcal{K}$-group hypothesis. We omit the details here, and instead appeal to \cref{MaxEssen} and a result in \cite{parkersem}.

Finally, we wish to apply \cref{G2Cor} to determine $\fs$. Except in the case where $q=p\in\{5,7\}$, we have that $Q_1, Q_2$ are the only possible essentials and $O^{p'}(\Out_{\fs}(Q_i))\cong\SL_2(q)$ for $i\in\{1,2\}$. In particular, the application of \cref{G2Cor} via \cref{MainThm} relies only on the classification of weak BN-pairs of rank $2$ provided in \cite{Greenbook} and again is independent of any $\mathcal{K}$-group hypothesis. We remark that there is currently no known way of determining whether a fusion system is exotic without appealing to the classification of finite simple groups, and instead appeal to \cite[Theorem 6.2]{parkersem} for a proof of the exoticity of the fusion systems listed in (vii).

\begin{theorem}
Let $\fs$ be a saturated fusion system over a Sylow $p$-subgroup of $\mathrm{G}_2(p^n)$ with $p\geq 5$. Then one of the following holds
\begin{enumerate}
\item $\fs=\fs_S(S: \Out_{\fs}(S))$;
\item $\fs=\fs_S(Q_1: \Out_{\fs}(Q_1))$ where $O^{p'}(\Out_{\fs}(Q_1))\cong\SL_2(q)$ or $q=p\in\{5,7\}$ and the possibilities for $O^{p'}(\Out_{\fs}(Q_1))$ are given in \cite[Lemma 5.2]{parkersem};
\item $\fs=\fs_S(Q_2: \Out_{\fs}(Q_2))$ where $O^{p'}(\Out_{\fs}(Q_2))\cong\SL_2(q)$;
\item $\fs=\fs_S(M)$ where $M\cong 5^3.\SL_3(5)$, $p=5$ and $n=1$;
\item $\fs=\fs_S(G)$ where $G\cong\mathrm{Ly}$, $\mathrm{HN}$, $\Aut(\mathrm{HN})$ or $\mathrm{B}$, $p=5$ and $n=1$;
\item $\fs=\fs_S(G)$ where $G\cong \mathrm{M}$, $p=7$ and $n=1$;
\item $p^n=7$ and, assuming CFSG, $\fs$ is one of the exotic fusion systems recorded in \cite[Table 5.1]{parkersem}; or
\item $\fs=\fs_S(G)$ where $F^*(G)=O^{p'}(G)\cong \mathrm{G}_2(p^n)$.
\end{enumerate}
\end{theorem}
\begin{proof}
Suppose first that there is an essential $E\not\in\{Q_1, Q_2\}$. By \cref{ExoticEssential}, $p=q=7$ and the action of $O^{7'}(\Out_{\fs}(E))$ is irreducible on $E$. In particular, since $O_7(\fs)$ is normal in $S$ and contained in each essential subgroup by \cref{normalinF}, $O_7(\fs)=\{1\}$. Then the hypothesis of \cite[Theorem 5.1]{parkersem} are satisfied and $\fs$ is one of the fusion systems described in \cite[Table 5.1]{parkersem}.

Hence, we may assume that $\mathcal{E}(\fs)\subseteq \{Q_1, Q_2\}$. If $\mathcal{E}(\fs)=\emptyset$, the (i) is satisfied. Suppose that $Q_2$ is essential and notice that $Z_3(S)=\Phi(Q_2)$. Since $[S, S']\le Z_3(S)$ and $S'$ has index $q$ in $Q_2$, it follows in a similar manner to \cref{SEFF} that $Q_2/\Phi(Q_2)$ is a natural $\SL_2(q)$-module for $O^{p'}(\Out_{\fs}(Q_2))\cong\SL_2(q)$. Moreover, since $S$ does not centralize $Z_2(S)=Z(Q_2)$ but acts quadratically on $Z(Q_2)$, it follows $Z(Q_2)$ is also a natural $\SL_2(q)$-module for $O^{p'}(\Out_{\fs}(Q_2))$ and since $S$ centralizes $Z_3(S)/Z_2(S)$, $O^{p'}(\Out_{\fs}(Q_2))$ centralizes $Z_3(S)/Z_2(S)$. In particular, if $Q_1$ is not essential then (iii) is satisfied.

Suppose that $Q_1$ is essential. Notice that $O^{p'}(\Out_{\mathrm{G}_2(q)}(Q_1))\cong\SL_2(q)$ acts irreducibly on $Q_1/\Phi(Q_1)$ and it follows that $\langle \Out_S(Q_1)^{\Out(Q_1)}\rangle$ acts irreducibly on $Q_1/\Phi(Q_1)$ and centralizes $\Phi(Q_1)$. Then by \cite[Lemma 2.73]{parkerSymp}, $\langle \Out_S(Q_1)^{\Out(Q_1)}\rangle$ is isomorphic to an irreducible subgroup of $\mathrm{Sp}_4(q)$ and so $O^{p'}(\Out_{\fs}(Q_1))$ is isomorphic to a subgroup of $\mathrm{Sp}_4(q)$ with a strongly $p$-embedded subgroup. Applying \cref{MaxEssen}, it follows that $O^{p'}(\Out_{\fs}(Q_1))$ is isomorphic to a central extension of $\PSL_2(q)$; or $q=p\in\{5,7\}$ and the possibilities are determined in \cite[Lemma 5.2]{parkersem}. If $Q_2$ is not essential then (ii) is satisfied.

If both $Q_1$ and $Q_2$ are essential, then since $O_p(\fs)\le Q_1\cap Q_2$ by \cref{normalinF} and $O^{p'}(\Out_{\fs}(Q_2))$ is irreducible on $Z_2(S)$ and $Q_2/Z_3(S)$, we have that $Z_2(S)\le O_p(\fs)\le Z_3(S)$ or $O_p(\fs)=\{1\}$. If $O_p(\fs)=\{1\}$, then $\fs$ is determined by \cref{G2Cor}, and the result holds. So suppose that $Z_2(S)\le O_p(\fs)\le Z_3(S)$. If $Z_2(S)=O_p(\fs)$, then $C_{Q_1}(Z_2(S))=S'$ is $\Aut_{\fs}(Q_1)$-invariant and since $Q_2$ centralizes $Z_2(S)$, $Q_1/S'$ and $Z_3(S)/Z_2(S)$, it follows from \cref{SEFF} that $S'/Z_2(S)$ is a natural module for $O^{p'}(\Out_{\fs}(Q_1))\cong \SL_2(q)$, and both $Z_2(S)$ and $Q_1/S'$ are centralized by $O^{p'}(\Out_{\fs}(Q_1))$. Letting $t\in Z(O^{p'}(\Out_{\fs}(Q_1)))$, by coprime action we have that for $V:=Q_1/Z(S)$, $V=[V, t]\times C_V(t)$ and $[V,t]$ is normalized by $S$. Since $Z_2(S)$ is centralized by $t$, we deduce that $[V, t]\cap Z(S/Z(S))=\{1\}$ so that $[V,t]=\{1\}$ and $t$ centralizes $V$, a contradiction. Therefore, $Z_2(S)<O_p(\fs)\le Z_3(S)$ so that $Z_3(S)=C_S(O_3(\fs))\le Q_1\cap Q_2$. Then by \cref{normalinF}, $C_S(O_3(\fs))\normaleq \fs$ and since $Z_3(S)$ is elementary abelian, $O_3(\fs)=Z_3(S)$.

Setting $L_1:=O^{p'}(\Out_{\fs}(Q_1))$, we have that $L_1/C_{L_1}(Q_1/Z_3(S)\cong \SL_2(q)$ and $L_1/C_{L_1}(Z_3(S)/Z(S))\cong \SL_2(q)$, and either $C_{L_1}(Q_1/Z_3(S)=C_{L_1}(Z_3(S)/Z(S))$ and $L_1\cong \SL_2(q)$; or $L_1$ is isomorphic to a central extension of $\PSL_2(q)$ by an elementary group of order $4$. Since $p\geq 5$, $\PSL_2(q)$ is perfect and has Schur multiplier of order $2$, and as $L_1=O^{p'}(L_1)$, we have a contradiction in the latter case. Therefore, $L_1\cong\SL_2(q)\cong O^{p'}(\Out_{\fs}(Q_2))$.

Now, $Z_3(S)$ is a normal, $S$-centric subgroup of $\fs$. By \cref{model}, there is a finite group $G$ such that $F^*(G)=Z_3(S)$ and $\fs=\fs_S(G)$. Moreover, $O^{p'}(\Out_G(Q_i))\cong \SL_2(q)$ and $\Out_{\fs}(Q_i)$ acts faithfully on $Q_i/Z_3(S)$ for $i\in\{1,2\}$. Set $\bar{G}:=G/Z_3(S)$ and notice that $\bar{Q_1}$ and $\bar{Q_2}$ are self-centralizing in $\bar{G}$. Moreover, $\bar{G}=\langle N_{\bar{G}}(\bar{Q_1}), N_{\bar{G}}(\bar{Q_2})\rangle$, and $\bar{Q_i}$ is $\Aut_{\bar{G}}(\bar{S})$-invariant for $i\in\{1,2\}$. It follows that $\bar{G}$ has a weak BN-pair of rank $2$ in the sense of \cite{Greenbook}. Moreover, since $Q_2$ centralizes $Z_2(S)$ which has index $q$ in $Z_3(S)$ and $Q_2/Z_3(S)$ is elementary abelian of order $q^2$, we deduce that $Z_3(S)$ is an FF-module for $\bar{G}$ by \cref{BasicFF}. Then, comparing with the completions in \cite{Greenbook} and applying \cite[Theorem A]{ChermakJ}, we conclude that $O^{p'}(\bar{G})\cong\SL_3(q)$ and $Z_3(S)$ is a natural module for $O^{p'}(\bar{G})$. As in the case when $p=2$, we observe that if $S$ splits over $Z_3(S)$, then $S$ is isomorphic to a Sylow $p$-subgroup of $\SL_4(q)$, which has $p$-rank $4n$ by \cite[Theorem 3.3.3]{GLS3}, whereas $S$ has $p$-rank $3n$. Therefore, $S$ is non-split and by \cite[Table I]{bell}, it follows that $q=p=5$. One can check that there is a unique fusion system up to isomorphism on $S$ with $O_5(\fs)=Z_3(S)$.
\end{proof}

\begin{remark}
In case (iv) of the above theorem, one can take $M$ to be a maximal subgroup of $\mathrm{Ly}$.
\end{remark}

\section[Fusion Systems on a Sylow \texorpdfstring{$p$}{p}-subgroup of \texorpdfstring{$\PSU_4(p^n)$}{PSU4(pn)}]{Fusion Systems on a Sylow $p$-subgroup of $\PSU_4(p^n)$}

We set $S$ to be a Sylow $p$-subgroup of $\PSU_4(q)$ where $q=p^n$ and $\fs$ to be a saturated fusion system supported on $S$. Again, let $\mathbb{K}$ be the finite field of order $q$ and recall the commutator formulas from \cref{G2Sylow}.

\begin{lemma}\label{Q1Unique}
There exists a unique subgroup $X:=X_\alpha X_{\alpha+\beta}X_{2\alpha+\beta}\le S$ of order $q^5$ such that $X'=Z(S)$, $|X|>q^4$, $S'=X\cap J(S)$ and $X$ is maximal by inclusion with respect to these properties. In particular, $X$ is characteristic in $S$.
\end{lemma}
\begin{proof}
By the definition of $X$, $|X|=q^5>q^4$ and $X\cap J(S)=S'$. Moreover, it follows from the commutator relations that $X'=Z(S)$. Thus, $X$ satisfies the required properties. Suppose there is $Y\not\le X$ such that $Y$ also satisfies the required properties. Since $Y\not\le X$ and $Y\cap J(S)=S'$, there is $y:=x_{\alpha}(t_1)x_{\beta}(t_2)\in Y$ with $t_1\ne 0\ne t_2$. By the requirements, $[Y, y]\le Y'=Z(S)$ and since $[y, x_{\alpha}(t)]\not\le Z(S)$ it follows that $Y\cap X=S'$. However, $|Y|>q^4$ so that $|XY|=|X||Y|/|X\cap Y|>q^6=|S|$, a clear contradiction.
\end{proof}

\begin{remark}
We may uniquely define $X$ as the preimage in $S$ of $J(S/Z(S))$. Moreover, $X$ is an ultraspecial special group with $Z(X)=X'=Z(S)$ of order $q$.
\end{remark}

We set $Q_1:=X$ and $Q_2:=J(S)$ with the intention of proving $\mathcal{E}(\fs)\subseteq\{Q_1, Q_2\}$. As it turns out, this is true except when $q=p=2$ where $S$ is coincidentally isomorphic to a Sylow $2$-subgroup of $\PSL_4(2)$. In this case, since $|S|=2^6$, we can directly compute that $S$-radical, $S$-centric subgroups of $S$ and classify all saturated fusion systems on $S$ with the aid of MAGMA.

\begin{proposition}
Let $S$ be isomorphic to a Sylow $2$-subgroup of $\mathrm{PSU}_4(2)$. If $X$ is an $S$-centric, $S$-radical subgroups of $S$ then either:
\begin{enumerate}
    \item $X\in\{S, Q_1, Q_2\}$;
    \item $X=C_S(x)$ for some $x\in S'\setminus Z(S)$ so that $|C_S(x)|=2^5$; or
    \item $X\in\mathcal{A}(Q_1)$ with $X\not\le Q_2$ so that $|X|=2^3$.
\end{enumerate}
\end{proposition}

\begin{proposition}\label{q=2Class2}
Let $\fs$ be a saturated fusion system over a Sylow $2$-subgroup of $\PSU_4(2)$. Then one of the following holds:
\begin{enumerate}
\item $\fs=\fs_S(S:\Out_{\fs}(S))$;
\item $\fs=\fs_S(Q_2:\Out_{\fs}(Q_2))$ where $\Out_{\fs}(Q_2)\cong \PSL_2(4)$;
\item $\fs=\fs_S(Q_1:\Out_{\fs}(Q_1))$ where $\Out_{\fs}(Q_1)$ is isomorphic to a subgroup of $\Sym(3)\times 3$;
\item $\fs=\fs_S(Q_x:\Out_{\fs}(Q_x))$ where $Q_x=C_S(x)$ for any $x\in S'\setminus Z(S)$, and $\Out_{\fs}(Q_x)\cong\Sym(3)$;
\item $\fs=\fs_S(M)$ where $M\cong 2^4: (\Sym(3)\times \Sym(3))$;
\item $\fs=\fs_S(M)$ where $M\cong 2^3: \PSL_3(2)$;
\item $\fs=\fs_S(G)$ where $G\cong \PSU_4(2)$; or
\item $\fs=\fs_S(G)$ where $G\cong\PSL_4(2)$.
\end{enumerate}
\end{proposition}

Henceforth, we suppose that $q>2$. Consider $Q_1, Q_2$ and their normalizers as subgroups of $\PSU_4(q)$. Then, as $\GF(p)$-modules, $Q_2$ is a natural $\Omega_4^-(q)$-module for $O^{p'}(\Aut_{\PSU_4(q)}(Q_2))\cong\PSL_2(q^2)$ while $Q_1/Z(Q_1)$ is the direct sum of two natural $\SL_2(q)$-modules for $O^{p'}(\Out_{\PSU_4(q)}(Q_1))\cong\SL_2(q)$. With this information, we can properly analyze the centralizers of elements in $S$.

\begin{lemma}\label{Q_2Omega}
Let $F\le S$ be such that $F\not\le Q_2$. Then one of the following occurs:
\begin{enumerate}
\item $[Q_2, F]=[Q_2, S]=S'$ and $C_{Q_2}(F)=C_{Q_2}(S)=Z(S)$;
\item $p=2$, $[Q_2, F]=C_{Q_2}(F)$ has order $q^2$ and $|FQ_2/Q_2|\leq q$; or
\item $p$ is odd, $|[Q_2, F]|=|C_{Q_2}(F)|=q^2$, $S'=[Q_2, F]C_{Q_2}(F)$, $Z(S)=C_{[Q_2, F]}(F)$ and $|FQ_2/Q_2|\leq q$.
\end{enumerate}
\end{lemma}
\begin{proof}
This is a restatement of \cref{Omega4}.
\end{proof}

\begin{lemma}\label{q^5cent}
Let $x\in S'\setminus Z(S)$. Then $Q_2\le C_S(x)$, $|C_S(x)|=q^5$, $Z(C_S(x))=C_{Q_2}(C_S(x))$ has order $q^2$ and $C_S(x)'=[Q_2, C_S(x)]$ has order $q^2$.
\end{lemma}
\begin{proof}
Let $x\in S'\setminus Z(S)$. Then since $x\in Q_2$, and $Q_2$ is elementary abelian, $Q_2\le C_S(x)$ so that $Q_2=J(S)=J(C_S(x))$ is characteristic in $C_S(x)$. Moreover, since $x\in Q_1\setminus Z(Q_1)$, we have that $|C_{Q_1}(x)|=q^4$. Then $C_{Q_1}(x)Q_2\le C_S(x)$ and so $|C_S(x)|\geq q^5$. Suppose $|C_S(x)|>q^5$. Then $q^6<|C_S(x)||Q_1|/|C_{Q_1}(x)|=|C_S(x)Q_1|\leq |S|=q^6$, a contradiction.

Since $Q_2$ is self-centralizing and $Q_2\le C_S(x)$, we have that $Z(C_S(x))=C_{Q_2}(C_S(x))$ may be determined from the information provided in \cref{Q_2Omega}. Indeed, since $x \in Z(C_S(x))\setminus Z(S)$, we have that $|[Q_2, C_S(x)]|=|Z(C_S(x))|=q^2$. Finally, it is clear from the commutator formulas that $C_S(x)'=[Q_2, C_S(x)]$, as required.
\end{proof}

\begin{lemma}
Let $x\in Q_2\setminus S'$. Then $C_S(x)=Q_2$.
\end{lemma}
\begin{proof}
Let $x\in Q_2\setminus S'$. Since $Q_2$ is abelian, $Q_2\le C_S(x)$ and $|C_S(x)|\geq q^4$. We have that $S'\le C_{Q_1}(x)$ so that $C_{Q_1/Z(S)}(x)$ is of order at least $q^2$. But $Q_1/Z(S)$, as a $\mathrm{GF}(p)\Out_{\PSU_4(q)}(Q_1)$-module, is a direct sum of natural $\SL_2(q)$-modules so that $|C_{Q_1/Z(S)}(x)|=q^2$ from which it follows that $S'=C_{Q_1}(x)$. Then $q^6=|S|\geq|C_S(x)Q_1|=|C_S(x)||Q_1|/|S'|\geq q^6$ so that $S=C_S(x)Q_1$, $|C_S(x)|=q^4$ and $C_S(x)=Q_2$.
\end{proof}

\begin{lemma}\label{q^4cent}
Let $x\in S\setminus Q_2$ be of order $p$. Then $C_S(x)\le Q_1$, $|C_S(x)|=q^4$, $|C_S(x)\cap Q_2|=q^2$, $m_p(C_S(x))\leq 3n$, $C_S(x)'=Z(S)$ and $|Z(C_S(x))|=q^2$.
\end{lemma}
\begin{proof}
Upon demonstrating that $C_S(x)\le Q_1$, the results follow from the structure of $Q_1$. Since $C_S(x)$ is centralized by $x\not\in Q_2$, it follows that $C_S(x)\cap Q_2\le S'$ and $C_S(x)S'$ has order $q^5$ and intersects $Q_2$ in $S'$. Hence, if $(C_S(x)S')'=Z(S)$, then $C_S(x)S'=Q_1$ by \cref{Q1Unique}. It is clear from \cref{Q_2Omega} that $[S', C_S(x)]=Z(S)$ and so it remains to show that $C_S(x)'\le Z(S)$. Indeed, since $S$ splits over $Q_2$, $C_S(x)$ splits over $S'$ and since $C_S(x)S'/S'$ is elementary abelian, we need only show that $[C_S(x)\cap S', C_S(x)]=Z(S)$. But this follows from \cref{Q_2Omega}, and the result is proved.
\end{proof}

With this information, we can determine the $S$-centric, $S$-radical subgroups of $S$, which we do over the following two propositions.

\begin{proposition}
\label{U4Centric1}
Suppose that $E$ is an $S$-centric, $S$-radical subgroup of $S$ and $S'\not\le E$. Then $E$ is elementary abelian of order $q^3$, $E\le Q_1$ and either
\begin{enumerate}
\item $p=2$, $E\normaleq S$ and $|E\cap S'|=q^2$;
\item $p$ is odd, $N_S(E)=Q_1$ and $|E\cap S'|=q^2$; or
\item $p$ is arbitrary, $N_S(E)=Q_1$ and $E\cap S'=Z(S)$.
\end{enumerate}
Moreover, in all cases, $E$ is not essential in any saturated fusion system $\fs$ over $S$.
\end{proposition}
\begin{proof}
Suppose that $S'\not\le E$. Since $[E,S']\le [S, S']\le Z(S)\le \Omega(Z(E))$, we must have that $[S', \Omega(Z(E))]\ne\{1\}$ for otherwise $S'$ centralizes the chain $\{1\}\normaleq \Omega(Z(E))\normaleq E$, a contradiction by \cref{Chain} since $S$-radical. Since $S'$ centralizes $Q_2$, there is $x\in \Omega(Z(E))$ with $x\in S\setminus Q_2$ and $E\le C_S(x)$. In particular, $Z(C_S(x))\le Z(E)$, $|Z(E)Q_2/Q_2|\geq q$ and $E\le Q_1$ by \cref{q^4cent}.

Suppose first that $E\cap S'>Z(S)$. Then for $e\in (E\cap S')\setminus Z(S)$, $Z(E)\le C_S(e)$. In particular, $|Z(E)Q_2/Q_2|=q$. Moreover, $C_{S'}(\Omega(Z(E)))=Z(C_S(e))$ has order $q^2$ and centralizes the chain $\{1\}\normaleq \Omega(Z(E))\normaleq E$ so that $C_{S'}(\Omega(Z(E)))= E\cap S'$ has order $q^2$. Suppose that $|EQ_2/Q_2|>q$. Then by \cref{Q_2Omega}, we have $Z(S)=[E, E\cap S']\le E'$ and either $E'=Z(S)$ and $Q_1$ centralizes the chain $\{1\}\normaleq E'\normaleq E$, a contradiction since $E$ is $S$-radical and $S'\not\le E$; or $Z(S)<E'\le E\cap S'$, $C_E(E')=E\cap C_S(e)=Z(E)(E\cap S')$ has order $q^3$ and $[E, C_E(E')]=[E, S\cap E']=Z(S)$ is characteristic in $E$ and again, $Q_1$ centralizes a characteristic chain. Thus, $|EQ_2/Q_2|=q$ and $E=Z(E)(E\cap S')$ is elementary abelian of order $q^3$. Since $E\le Q_1$ and $Q_1'=Z(S)\le E$, we deduce that $E\normaleq Q_1$. Moreover, when $p=2$, it follows from \cref{Q_2Omega} that $[C_S(e), E]\le C_S(e)'=(S'\cap E)$ and so $E\normaleq S=Q_1C_S(e)$.

Suppose now that $E\cap S'=Z(S)$. Since $E\le Q_1$, it follows that $E\cap Q_2=Z(S)$ and $|E|\leq q^3$. If $\Omega(Z(E))\le Q_2$, then $\Omega(Z(E))=Z(S)$ and so $Q_1$ centralizes the chain $\{1\}\normaleq \Omega(Z(E))\normaleq E$, a contradiction since $E$ is $S$-radical. Hence, there is $e\in\Omega(Z(E))\setminus Q_2$ and so, $E\le C_S(e)$. Since $E$ is $S$-centric, we must have that $Z(C_S(e))\le \Omega(Z(E))$. If $\Omega(Z(E))=Z(C_S(e))$, then as $C_S(e)'=Z(S)$, $C_S(e)$ centralizes the chain $\{1\}\normaleq \Omega(Z(E))\normaleq E$, and since $E$ is $S$-radical, $E=C_S(e)$. But then $Q_1$ centralizes the chain $\{1\}\normaleq E'\normaleq E$, a contradiction. So there is $e'\in \Omega(Z(E))\setminus(Q_2C_S(e))$ with $Z(C_S(e'))\cap Z(C_S(e))=Z(S)$ and $Z(C_S(e'))\le \Omega(Z(E))$. In particular, $Z(C_S(e'))Z(C_S(e))$ is an  elementary abelian subgroup of $E$ of order $q^3$, and since $E$ itself has order at most $q^3$, we conclude that $E=Z(C_S(e))Z(C_S(e'))$. Then for any $y\in Q_2\setminus S'$, $[E, y]\not\le Z(S)$ and so $N_{Q_2}(E)=S'$. Since $E\le Q_1$ and $Q_1'=Z(S)\le E$, we have that $N_S(E)=Q_1$.

Suppose that for any of the $E$ considered, $E$ is essential is some saturated fusion system $\fs$ supported on $S$. Suppose first that we are in case (i) or (ii). Then $S'$ centralizes $E\cap S'$ and since $|S'/E\cap S'|=|E/E\cap S'|=q$, it follows from \cref{SEFF} that $O^{p'}(\Out_{\fs}(E))\cong\SL_2(q)$ and $\Out_{S'}(E)\in\syl_p(E)$. But $|N_S(E)/E|\geq q^2$ in either case, a contradiction. Hence, we may assume that we are in case (iii) and $E\cap S'=Z(S)$. Let $e\in E\setminus Q_2$ so that $E\le C_S(e)$, where $|C_S(e)|=q^4$. Then $Z(C_S(e))$ is a subgroup of $E$ of index $q$ centralized by $C_S(e)$ where $|C_S(e)E/E|=q$ and $C_S(e)\le N_S(E)=Q_1$. By \cref{SEFF}, $O^{p'}(\Out_{\fs}(E))\cong\SL_2(q)$ and $\Out_{C_S(e)}(E)\in\syl_p(E)$, and since $|N_S(E)/E|=q^2$, we have another contradiction.
\end{proof}

\begin{proposition}\label{U4Centric2}
Suppose that $E$ is an $S$-centric, $S$-radical subgroup of $S$, $S'\le E$ and $q>2$. Then $E\in\{Q_1, Q_2, S\}$.
\end{proposition}
\begin{proof}
Since $S'\le E$, we have that $Z(E)\le Q_2$. Moreover, if $E\le Q_2$, then using that $E$ is $S$-centric, we conclude that $E=Q_2$. So we may suppose throughout the remainder of this proof that there is $e\in E\setminus Q_2$.

Suppose first that $Z(E)=Z(S)$ so that $S'\le Z_2(E)$. Indeed, if $E\cap Q_2>S'$, then it follows from the commutator formulas that $Z_2(E)=S'$ and $S$ centralizes the chain $\{1\}\normaleq Z(E)\normaleq Z_2(S)\normaleq E$, and since $E$ is $S$-radical, we deduce that $E=S$. So if $Z(E)=Z(S)$, then we may assume that $E\cap Q_2=S'$. 

In addition, suppose that $E'=Z(S)$. Consider $A\in\mathcal{A}(E)$. Since $S'\le E$ and $S'$ is elementary abelian, we infer that $|A|\geq 3n$. Moreover, there is $a \in A$ with $a\not\le Q_2$, else $S'=J(E)$ and $Q_2$ centralizes the chain $\{1\}\normaleq J(E)\normaleq E$, a contradiction since $E$ is $S$-radical. It follows that $A\le C_S(a)\le Q_1$, $|A|=q^3$ and $|A\cap S'|=q^2$. Then either $E=AS'\le Q_1$; or $|E|>q^4$. In either case, it follows from \cref{Q1Unique} that $E\le Q_1$ and then $Q_1$ centralizes the chain $\{1\}\normaleq Z(E)\normaleq E$. Since $E$ is $S$-radical, $Q_1\le E$. Since $E\cap Q_2=S'$, it follows from a consideration of orders that $E=Q_1$.

Suppose that $Z(S)=Z(E)<E'$. By \cref{q^5cent}, $C_E(E')\le C_S(x)$ for some $x\in E'\setminus Z(E)$ and it follows that either $C_E(E')=S'$; or $C_E(E')\not\le Q_2$ and $Z(C_E(E'))\le S'$ has order $q^2$. In the former case, $S$ centralizes the chain $\{1\}\normaleq Z(E)\normaleq C_E(E')\normaleq E$, and since $E$ is $S$-radical, $E=S$, a contradiction since $E\cap Q_2=S'$. Therefore, $C_E(E')\not\le Q_2$ and since $C_E(E')\cap Q_2\le E\cap Q_2=S'$, we conclude that $|C_E(E')|\leq q^4$.

Let $A\in\mathcal{A}(C_E(E'))$ and suppose that $A\cap S'>Z(C_E(E'))$. Comparing with the commutator formulas, it follows that $A\le C_S(A\cap S')=S'$ and so $A=S'$. Notice that if $S'=J(C_E(E'))$, then $Q_2$ centralizes the chain $\{1\}\normaleq S'\normaleq E$, a contradiction since $E$ is $S$-radical. Thus, we may assume that there is $A\in\mathcal{A}(C_E(E))$ with $A\cap S'=Z(C_E(E'))$ and $|A|\geq q^3$. In particular, $C_E(E')=AS'$ and $|A|=q^3$. Then for $a\in A\setminus A\cap S'$, we infer that $A\le C_S(a)\le Q_1$ and so $C_E(E')\le Q_1$. But now, since $S'\le C_E(E')$, $Q_1$ centralizes the chain $\{1\}\normaleq Z(E)\normaleq C_E(E')\normaleq E$, a contradiction since $|E|\leq q^5$, $E$ is $S$-radical and $E'>Z(S)$.

Suppose now that $Z(S)<Z(E)$. Since $E\not\le Q_2$, $Z(E)\le S'$ and $E\le C_S(x)$ for some $e\in Z(E)\setminus Z(S)$. Since $E$ is $S$-centric,  $Z(C_S(x))\le Z(E)$ and since $E\not\le Q_2$, it follows from \cref{q^5cent}, that $Z(C_S(x))=Z(E)$. Indeed, if $p=2$, then $Z(E)=C_{Q_2}(E)=[Q_2, E]$ and $Q_2$ centralizes the chain $\{1\}\normaleq Z(E)\normaleq E$. Since $E$ is $S$-radical, $Q_2=J(E)$ is characteristic in $E$. Then, $[C_S(x), E]\le J(E)$ and $Z(E)=[J(S), C_S(x)]$ and $C_S(x)$ centralizes the chain $\{1\}\normaleq Z(E)\normaleq J(E)\normaleq E$, and since $E$ is $S$-radical, $E=C_S(x)$. Now, assuming $q>2$, both $Z(S)$ and $S'$ are characteristic subgroups of $E$ by \cite[Lemma 3.13]{ParkU4}. Then $S$ centralizes the chain $\{1\}\normaleq Z(S)\normaleq S'\normaleq E$, a contradiction since $E$ was assumed to be $S$-radical.

Suppose now that $p$ is odd and $Z(C_S(x))=Z(E)$. Let $A\in\mathcal{A}(E)$ such that $A\not\le Q_2$. Then, there is $a\in A$ such that $|C_S(a)|=q^4$, $A\le C_S(a)\cap C_S(x)$, $C_S(a)\le Q_1$ and $Z(E)=C_S(a)\cap S'$. Now, $|C_S(x)\cap C_S(a)|=q^3$ and it follows that any elementary abelian subgroup of $E$ not contained in $Q_2$ has order at most $q^3$. Since $E\cap Q_2$ is elementary abelian, it follows that either $J(E)=E\cap Q_2\ge S'$, or $E\cap Q_2=S'$ and there is $A\in\mathcal{A}(E)$ with $|A|=q^3$ and $A\cap S'=Z(E)$. In the latter case, it follows that $E=AS'$ has order $q^4$ and since $A\le C_S(a)\le Q_1$, we have that $E\le Q_1$. Moreover, $E'=[A, S']=Z(S)$ and $Q_1$ centralizes the chain $\{1\}\normaleq E'\normaleq E$, a contradiction since $E$ is $S$-radical. Thus, $J(E)=E\cap Q_2$ and so $Q_2$ centralizes the chain $\{1\}\normaleq J(E)\normaleq E$, and since $E$ is $S$-radical, $Q_2=J(E)$. But then, since $p$ is odd, $S'=[Q_2, E]Z(E)$, $Z(S)=[Q_2, E]\cap Z(E)$ and $S$ centralizes the chain $\{1\}\normaleq Z(S)\normaleq S'\normaleq E$, a contradiction since $Z(E)>Z(S)$ and $E$ is $S$-radical.
\end{proof}

We now complete the classification of saturated fusion systems supported on a Sylow $p$-subgroup of $\PSU_4(p^n)$. When $q=p$ we get some exceptional behaviour, particularly when $p=3$, and refer to \cite{Baccanelli} and \cite{Raul} where these cases have already been treated. Additionally, by \cref{U4Centric2}, we have that $\mathcal{E}(\fs)\subseteq \{Q_1, Q_2\}$. 

As in earlier sections in this chapter, we endeavor to classify saturated fusion systems on $S$ without the need for a $\mathcal{K}$-group hypothesis. When $p=2$, since $m_2(S/Q_i)>1$, \cite{Bender} provides a list of groups with a strongly embedded subgroups, and so we focus more on the case where $p$ is odd. Here, $Q_1/\Phi(Q_1)$ witnesses quadratic action by $S$, and we rely on results of Ho (although we believe it should be possible to find a more elementary proof) to show that $O^{p'}(\Out_{\fs}(Q_1))\cong \SL_2(q)$. With regards to $Q_2$, we come up short and rely on $\mathcal{K}$-group hypothesis to identify $O^{p'}(\Out_{\fs}(Q_2))$ with $\PSL_2(q^2)$. We believe this can be achieved without using a $\mathcal{K}$-group hypothesis as follows: 

By the conditions on $G:=O^{p'}(\Out_{\fs}(Q_2)$, we see quickly that $\syl_p(G)$ is a TI-set for $G$. Then, using some appropriately chosen minimality condition, we should be able to prove that $G=\langle S, T\rangle$ and $C_{Q_2}(S)\cap C_{Q_2}(T)=\{1\}$ for any $S, T\in\syl_p(G)$. Even better, $C_{Q_2}(S)\cap [Q_2, T]=\{1\}$ for all such $S$ and $T$. Noticing that $|Q_2/C_{Q_2}(S)|=q^3$, we strive to show that $Q_2/C_{Q_2}(S)=[Q_2/C_{Q_2}(S), S]\cup\bigcup_{s\in S} C_{Q_2}(T^s)C_{Q_2}(S)/C_{Q_2}(S)$, where the intersection of any of the two subgroups in the union is $C_{Q_2}(S)$. 

Then, we aim to show that $C_{Q_2}(S)$ and $C_{Q_2}(T)$ are the only centralizers of a Sylow $p$-subgroup of $G$ contained in $C_{Q_2}(T)C_{Q_2}(S)$, for then we have a correspondence between Sylow $p$-subgroups of $G$ and certain subgroups of $Q_2$ of order $q$ and, more importantly, we are able to deduce that there are only $q^2+1$ Sylow $p$-subgroups of $G$. We are then in a position to recognize $\PSL_2(q^2)$ via a result of Hering, Kantor and Seitz \cite{SplitBN} which recognizes a split BN-pair of rank $1$ in $G$ .

In the classification of fusion systems supported on $S$, we apply \cref{PSUCor} using \cref{MainThm} when $Q_1$ and $Q_2$ are both essential and, as in earlier cases, we remark that this reduces to applying the main result from \cite{Greenbook}, which is independent of any $\mathcal{K}$-group hypothesis.

\begin{theorem}
Let $\fs$ be a saturated fusion system over a Sylow $p$-subgroup of $\PSU_4(p^n)$ for $p^n>2$. Moreover, if $p$ is odd then assume that $O^{p'}(\Out_{\fs}(Q_2))$ is a $\mathcal{K}$-group. Then one of the following occurs:
\begin{enumerate}
\item $\fs=\fs_S(S: \Out_{\fs}(S))$;
\item $\fs=\fs_S(Q_1: \Out_{\fs}(Q_1))$ where $O^{p'}(\Out_{\fs}(Q_1))\cong \SL_2(p^n)$, or $p^n=p=3$ and $\Out_{\fs}(Q_1)$ is determined in \cite{Baccanelli};
\item $\fs=\fs_S(Q_2: \Out_{\fs}(Q_2))$ where $O^{p'}(\Out_{\fs}(Q_2))\cong\PSL_2(p^{2n})$;
\item $\fs=\fs_S(G)$ where $G=\mathrm{Co}_2$, $\mathrm{McL}$, $\Aut(\mathrm{McL})$, $\PSU_6(2)$ or $\PSU_6(2).2$ and $p^n=3$; or
\item $\fs=\fs_S(G)$ where $F^*(G)=O^{p'}(G)\cong\PSU_4(p^n)$.
\end{enumerate}
\end{theorem}
\begin{proof}
Set $q=p^n$ throughout. If neither $Q_1$ nor $Q_2$ are essential then $\fs=\fs_S(S: \Out_{\fs}(S))$ and (i) holds. Suppose that $Q_1$ is essential and assume first that $q=p$. If $p=3$, then the action of $\Out_{\fs}(Q_1)$ on $Q_1$ is determined completely in \cite{Baccanelli} while if $p\geq 5$, then the action of $O^{p'}(\Out_{\fs}(Q_1))$ is determined by \cite[Lemma 4.4]{Raul}.

Suppose now that $Q_1$ is essential and $q>p$. If $p=2$, then as $m_p(S/Q_1)>1$, it follows from \cref{MaxEssen} that $O^{p'}(\Out_{\fs}(Q_1))\cong\SL_2(q)$. So suppose that $p$ is odd. Let $T,P\in\syl_p(O^{p'}(\Out_{\fs}(Q_1)))$ and suppose that $1\ne x \in T\cap P$. Notice that $Z(Q_1)=Z(S)$ so that $O^{p'}(\Out_{\fs}(Q_1))$ acts trivially on $Z(Q_1)$. Then $[Q_1, T]Z(Q_1)=[Q_1, x]Z(Q_1)=[Q_1, P]Z(Q_1)$ and $[Q_1, T, T]\le Z(Q_1)\ge [Q_1, P,P]$. It follows that $\langle P, T\rangle$ centralizes a series $\{1\}\normaleq Z(Q_1)\normaleq [Q_1, T]Z(Q_1)\normaleq Q_1$ and by \cref{GrpChain}, $\langle P, T\rangle$ is a $p$-group. Since $T,P\in\syl_p(O^{p'}(\Out_{\fs}(Q_1)))$, we must have that $T=P$. Moreover, $T$ acts quadratically on $Q_1/Z(Q_1)=Q_1/\Phi(Q_1)$ and so, by \cite[Theorem 1]{HoTI}, $O^{p'}(\Out_{\fs}(Q_1))$ is isomorphic to a central extension of $\PSL_2(q)$. Then eliminating $\PSL_2(q)$ since $T$ acts quadratically (see \cite[(I.3.8.4)]{gor}), we deduce that $O^{p'}(\Out_{\fs}(Q_1))\cong\SL_2(q)$. By \cref{DirectSum} and as $T\in\syl_p(O^{p'}(\Out_{\fs}(Q_1)))$, we conclude that $Q_1/Z(Q_1)$ is a direct sum of two natural $\SL_2(q)$-modules.

Suppose that $Q_2$ is essential. Since $S/Q_2$ is elementary abelian of order $q^2$ and $q>p$, it follows from \cref{MaxEssen} that $O^{p'}(\Out_{\fs}(Q_2))$ is isomorphic to a central extension of $\PSL_2(q^2)$. Then, since $S$ does not act quadratically on $Q_2$ and $Q_2$ contains a non-central chief factor, by \cref{SL2ModRecog}, we conclude that $Q_2$ is a natural $\Omega_4^-(q)$-module for $O^{p'}(\Out_{\fs}(Q_2))\cong \PSL_2(q)$, as required.

If both $Q_1$ and $Q_2$ are essential, then by \cref{normalinF}, $O_p(\fs)\le Q_1\cap Q_2$ and $O_p(\fs)$ is normalized by $O^{p'}(\Out_{\fs}(Q_2))$. Thus, $O_p(\fs)=\{1\}$ and since $Q_1$ and $Q_2$ are characteristic in $S$ and we satisfy the hypotheses of \cref{PSUCor}.
\end{proof}

\printbibliography

\end{document}